\documentclass[11pt, a4paper,reqno]{amsart}
\usepackage[colorlinks, citecolor=blue, urlcolor=blue]{hyperref}
\usepackage{amsthm, amssymb,amsmath,amsfonts,latexsym,mathrsfs}
\usepackage{bm}
\usepackage[us,12hr]{datetime}
\usepackage{mathtools}
\usepackage{enumerate}
\usepackage{array,graphics,color}
\allowdisplaybreaks
\usepackage{soul}

\usepackage{array,graphics,color}
\numberwithin{equation}{section}

\newtheorem{thm}{Theorem}[section]
\newtheorem{defn}[thm]{Definition}
\newtheorem{lem}[thm]{Lemma}

\newtheorem{cor}[thm]{Corollary}

\newtheorem{rem}[thm]{Remark}

\DeclarePairedDelimiterX{\norm}[1]{\lVert}{\rVert}{#1}
\DeclarePairedDelimiterX{\bnorm}[1]{\big\lVert}{\big\rVert}{#1}
\DeclarePairedDelimiterX{\Bnorm}[1]{\Big\lVert}{\Big\rVert}{#1}
\newcommand\at[2]{\left.#1\right|_{#2}}

\newcommand{\hil}{\mathcal{H}}
\newcommand{\hils}{\mathcal{B}_2(\hil)}
\newcommand{\boh}{\mathcal{B}_1(\hil)}
\newcommand{\dds}{\dfrac{d}{ds}}
\newcommand{\ddt}{\dfrac{d}{dt}}

\newcommand{\cir}{\mathbb{T}}

\usepackage[]{graphicx}
\begin{document}



\pagespan{1}{}


\title[\bf Second order trace formulae]{\bf Second order trace formulas}

\author[\bf Chattopadhyay] {\bf Arup Chattopadhyay}
\address{Department of Mathematics, Indian Institute of Technology Guwahati, Guwahati, 781039, India}
\email{arupchatt@iitg.ac.in, 2003arupchattopadhyay@gmail.com}

\author[Das]{Soma Das}
\address{Theoretical Statistics and Mathematics Unit, Indian Statistical Institute, Bangalore Centre, Bengaluru, 560059, India}	\email{dsoma994@gmail.com, soma18@iitg.ac.in, somadas\_ra@isibang.ac.in}

\author[Pradhan]{Chandan Pradhan}
\address{Department of Mathematics, Indian Institute of Science, Bangalore, 560012, India}
\email{chandanp@iisc.ac.in, chandan.pradhan2108@gmail.com}

\begin{abstract}
	Koplienko \cite{Ko} found a trace formula for perturbations of self-adjoint operators by operators of Hilbert-Schmidt class $\mathcal{B}_2(\mathcal{H})$. Later, Neidhardt introduced a similar formula in the case of pair of unitaries $(U,U_0)$ via multiplicative path in \cite{NH}. In 2012, Potapov and Sukochev \cite{PoSu} obtained a trace formula like the Koplienko trace formula for pairs of contractions by answering an open question posed by Gesztesy, Pushnitski, and Simon in \cite[Open Question 11.2]{GePu}. In this article, we supply a new proof of the Koplienko trace formula in the case of pairs of contractions $(T,T_0)$, where the initial operator $T_0$ is normal, via linear path by reducing the problem to a finite-dimensional one as in the proof of Krein's trace formula by Voiculescu \cite{Voi}, Sinha and Mohapatra \cite{MoSi94,MoSi96}. Consequently, we obtain the Koplienko trace formula for a class of pairs of contractions using the Sch\"{a}ffer matrix unitary dilation. Moreover, we also obtain the Koplienko trace formula for a pair of self-adjoint operators and maximal dissipative operators using the Cayley transform. At the end, we extend the Koplienko-Neidhardt trace formula for a class of pairs of contractions $(T,T_0)$ via multiplicative path using finite-dimensional approximation method.
\end{abstract}

\maketitle

\section{\bf Introduction}
In noncommutative geometry \cite{Connes,CC97}, the spectral action has been described in terms of trace of $f(H_0)$, where $f$ is a nice scalar function and $H_0$ is an unbounded self-adjoint operator having compact resolvent in a separable Hilbert space $\hil$. The main reason behind the curiosity of the trace function $V\mapsto \text{Tr}~(f(H_0+V))$ lies in the perturbation of $H_0$ by $V$, where $V$ is a bounded self-adjoint operator on $\hil$. For example, one can consider the inner fluctuations of a spectral triple, they play a significant role in the applications of  noncommutative geometry to high energy physics \cite{CC06,CC96,CC97,CCM07}. The spectral shift function for a trace class perturbation of a self-adjoint (unitary) operator and the associated trace formula plays an important role in perturbation theory. The notion of first order spectral shift function originated from Lifshits' work on theoretical physics \cite{Lif} and later the mathematical theory of this object was elaborated upon by Krein in a series of papers, starting with \cite{Kr53}. In \cite{Kr53} (see also \cite{Kr83}), Krein proved that for a pair
of self-adjoint (not necessarily bounded) operators $H$ and $H_0$ satisfying $H-H_0\in \mathcal{B}_1(\mathcal{H})$ (set of trace class operators on a separable Hilbert space $\mathcal{H}$) there exists a unique real valued $L^1(\mathbb{R})$ function $\xi$  such that 
\begin{equation}\label{inteq1}
	\text{Tr}~\{\phi(H)-\phi(H_0)\} = \int_{\mathbb{R}} \phi'(\lambda)~\xi(\lambda)~d\lambda,
\end{equation}
whenever $\phi$ is a function on $\mathbb{R}$ with the Fourier transform of $\phi'$ in $L^1(\mathbb{R})$. The function $\xi$ is known as Krein's spectral shift function and the relation \eqref{inteq1} is called Krein's trace formula. A similar result was obtained by Krein in \cite{Kr62} for pair of unitary operators $\big\{U,U_0\big\}$ such that $U-U_0$ is trace class. For each such pair there exists a real valued $L^1([0,2\pi])$- function $\xi$, unique modulo an additive constant, (called a spectral shift function for $\{U,U_0\}$) such that
\begin{equation}\label{intequ2}
	\text{Tr}~\big\{\phi(U)-\phi(U_0)\big\} = \int_{0}^{2\pi} \frac{d}{dt}\big\{\phi(\textup{e}^{it})\big\}~\xi(t)~dt,
\end{equation}
whenever $\phi'$ has absolutely convergent Fourier series. 
The original proof of Krein uses analytic function theory. Later in \cite{BS2} (see also \cite{BS1}), Birman and Solomyak approached the trace formula \eqref{inteq1} using the theory of double operator integrals. In 1985, Voiculescu \cite{Voi} gave an alternative proof of the trace formula \eqref{inteq1} by adapting the proof of the classical Weyl-von Neumann theorem for the case of bounded self-adjoint operators, and later Sinha and Mohapatra extended Voiculescu's method to the unbounded self-adjoint \cite{MoSi94} and unitary cases \cite{MoSi96}. In this connection it is worth mentioning that Potapov, Sukochev and Zanin \cite{PoSuZa14JST} gave an alternative proof of Lifshits'-Krein trace formula \eqref{inteq1} in {semi-finite von Neumann algebra setting}.

 We wanted to note that in \cite{MN2015}, Malamud, Neidhardt, and in \cite{MNP2017}, Malamud, Neidhardt, and Peller, also presented two other proofs of Krein's formula \eqref{inteq1} for a pair of self-adjoint operators. Moreover, in \cite{MN2015}, a simple formula for the spectral shift function for a pair of two self-adjoint extensions is expressed via the Weyl function and boundary operators. In \cite{Pe16}, Peller, and in \cite{AP16}, Aleksandrov and Peller provided precise descriptions of the maximal class of functions for which \eqref{inteq1} and \eqref{intequ2} held. This solved a longstanding problem by Krein.

The first attempt to generalize Krein's trace formulas to pairs of non-selfadjoint and non-unitary operators goes back to Langer \cite{Lan1965}. Note, however, that applied to a pair of contractions $(T,T_0)$, Langer's condition requires the strict contractivity of both operators. In particular, this result cannot be applied to a pair of unitary (self-adjoint) operators.

Later on, Rybkin \cite{Ry1,Ry2,Ry3,Ry4} considered a pair $(H, H_0)$ $((T, T_0))$ consisting of an m-dissipative operator $H$ and self-adjoint $H_0 = H_0^*$ (a unitary operator $T_0$ and contraction $T$) and proved trace formula similar to Krein's one \eqref{inteq1}. However, he proved the existence of a complex-valued $A$-integrable spectral shift function in the sense of Kolmogorov. Regarding the definition of $A$-integrable function and its properties, we refer to \cite{Al81, MNP19}.

Studying the accumulative case $(H,H_0)$ where $H := H_0-iG,~ H_0 = H_0^*$ and $G (\geq 0)\in \boh$, Krein used the class $K(\mathbb{R}_+)$ of functions holomorphic in $\mathbb{C}_-$ (lower-half plane) being Laplace transform of a bounded measure on $\mathbb{R}_+$. He proved in \cite[Theorem 9.2]{K1987} that $\phi(H)-\phi(H_0)\in \boh$ for any $\phi\in K(\mathbb{R}_+)$ and instead of \eqref{inteq1}, the following formula holds
\begin{align}\label{nn1}
	\operatorname{Tr}\left(\phi(H)-\phi(H_0) \right)=-i\int_{\mathbb{R}}\phi'(\lambda)d\omega_K(\lambda),
\end{align}
where $d\omega_K(\lambda)$ is a non-negative measure, and its absolute continuity was not addressed in \cite{K1987}.

The next step was made by Adamjan and Neidhardt in \cite{AdNei}. Namely, for pairs $(H, H_0): = (H_0-iG,H_0), ~G \geq 0$, with $G\log G \in\boh$, it is proved in \cite{AdNei} that under these assumptions there exists a real-valued spectral shift function $\xi$ such that instead of \eqref{nn1} the following formula holds
\begin{align}\label{nn2}
	\operatorname{Tr}\left(\phi(H)-\phi(H_0) \right)=\int_{\mathbb{R}}\phi'(\lambda)\xi(\lambda)\,d\lambda
\end{align}
for functions $\phi$ from a certain class of holomorphic  functions in $\mathbb{C}$, which is smaller than $K(\mathbb{R}_+)$. Note that the condition $G\log G\in\boh$ is stronger than $G\in\boh$.

Finally, the existence of a complex-valued integrable spectral shift function $\xi$ for a pair of contractions $(T, T_0)$ with trace class difference $T-T_0\in\boh$ was proved in \cite{MN2015} (under an additional assumption $\rho(T_0)\cap \mathbb{D}\neq\emptyset$) by Malamud and Neidhardt, and in \cite{MNP2017} and \cite{MNP19} (in full generality)  by Malamud, Neidhardt, and Peller. This settles the longstanding problem of obtaining a trace formula for arbitrary pairs of contractions with trace class differences as well as proving the existence of an integrable spectral shift function for such contractions. Moreover, the formula \eqref{nn2} for pair of contractions was proved in these paper for holomorphic operator Lipshitz functions in $\mathbb{D}$.

The modified second-order spectral shift function for Hilbert-Schmidt perturbations was introduced by Koplienko in \cite{Ko}. Let $H$ and $H_0$ be two self-adjoint operators in a separable Hilbert space $\mathcal{H}$ such that $H-H_0=V\in \mathcal{B}_2(\mathcal{H})$ (set of Hilbert-Schmidt operators on $\mathcal{H}$).  Sometimes $H_0$ is known as the initial operator, $V$ is known as the perturbation operator, and $H=H_0+V$ is known as the final operator.  In this case, the difference $\phi(H)-\phi(H_0)$ is no longer of trace-class, and one has to consider instead
\[
\phi(H)-\phi(H_0)-\at{\dfrac{d}{ds}\Big(\phi(H_0+sV)\Big)}{s=0},
\]
where $\at{\dfrac{d}{ds}\Big(\phi(H_0+sV)\Big)}{s=0}$ denotes the G\^{a}teaux derivative of $\phi$ at $H_0$ in the direction $V$ (see \cite{RB}) and find a trace formula for the above expression under certain assumptions on $\phi$. Under the above hypothesis, Koplienko's formula asserts that there exists a unique non-negative function $\eta \in L^1(\mathbb{R})$ such that
\begin{equation}\label{intequ3}
	\operatorname{Tr}\Big\{\phi(H)-\phi(H_0)-\at{\dfrac{d}{ds}\Big(\phi(H_0+sV)\Big)}{s=0}\Big\}=\int_{\mathbb{R}} \phi''(\lambda)~\eta(\lambda)~d\lambda
\end{equation}
for rational functions $\phi$ with poles off $\mathbb{R}$. The function $\eta$ is known as Koplienko spectral shift function corresponding to the pair $(H_0,H)$.  In 2007, Gesztesy et al. \cite{GePu} gave an alternative proof of the formula \eqref{intequ3} for the bounded case and Boyadzhiev \cite{BO} in 1993 and then Dykema and Skripka \cite{DS,SK10} in 2009, obtained the formula \eqref{intequ3} in the semi-finite von Neumann algebra setting. Later in 2012, Sinha and the first author \cite{ChSi} of this article  provide an alternative proof of the formula \eqref{intequ3} using the idea of finite dimensional approximation method as in the works of Voiculescu \cite{Voi}, Sinha and Mohapatra \cite{MoSi94,MoSi96}. Furthermore, in the context of Krein and Koplienko trace formulas on normed ideals in semifinite von Neumann algebra setting, the authors of \cite{DySk14CMP, CGP2023} used the concept of unitary dilation (only use the existence of unitary dilation but not the structure of the Sch\"{a}ffer matrix representation of unitary dilation) to achieve some estimates for traces. In this connection, it is worth mentioning that the existence of higher order spectral shift function for pair of self-adjoint operators $(H,H_0)$ with Schatten-n-class difference was obtained by Potapov, Skripka, and Sukochev in \cite{PoSkSu13In} using the concept of multiple operator integrals (MOI), which in particular, affirmatively answers the Koplienko's conjecture.

A similar problem for unitary operators was considered by Neidhardt \cite{NH}. Let $U$ and $U_0$ be two unitary operators on a separable Hilbert space $\hil$ such that $U-U_0~\in~\mathcal{B}_2(\hil)$. Then  $U=e^{iA}U_0$, where $A$ is a self-adjoint operator in $\mathcal{B}_2(\hil)$. Note that we interpret $U_0$ as the initial operator, $A$ as the perturbation operator, and $U=e^{iA}U_0$ as the final operator. Set $U_s=e^{isA}U_0,~s\in \mathbb{R}$. Then it was shown in \cite{NH} that there exists an $L^1([0,2\pi]))$-function $\eta$ (unique up to an additive constant ) such that 
\begin{equation}\label{intequ4}
	\operatorname{Tr}\Big\{\phi(U)-\phi(U_0)-\at{\dfrac{d}{ds}\big\{\phi(U_s)\big\}}{s=0}\Big\}=\int_{0}^{2\pi} \dfrac{d^2}{dt^2} \big\{\phi (e^{it})\big\} \eta(t) dt,
\end{equation}
whenever $\phi''$ has absolutely convergent Fourier series. The function $\eta$ is known as Koplienko spectral shift function corresponding to the pair $(U_0,U)$. In addition, the authors of this article also provide an alternative proof of the formula \eqref{intequ4} using the idea of the finite dimensional approximation method in \cite{CDP}. It is worth mentioning that, in \cite{Pe05}, Peller extended formulae \eqref{intequ3} and \eqref{intequ4} to the case when $\phi$ belongs to the Besov class $B_{\infty,1}^2$. It is important to note that the path considered by Neidhardt \cite{NH} is an unitary path, that is  $U_s=e^{isA}U_0$ is an unitary operator for each $s\in\mathbb{R}$. In \cite[Sect.10]{GePu}, Gesztesy, Pushnitski and Simon have discussed an alternative to Neidhardt's approach. In other words they have considered the linear path $U_0+ t(U-U_0);~0\leq t\leq 1$ instead of the unitary path $U_s=e^{isA}U_0;~0\leq s\leq 1$ and proved that there exists a real distribution $\eta$ on the unit circle $\mathbb{T}$ so that the formula
\begin{equation}\label{linshif}
	\operatorname{Tr}\Big\{\phi(U)-\phi(U_0)-\at{\dfrac{d}{ds}\Big\{\phi\big(U_0+s(U-U_0)\big)\Big\}}{s=0}\Big\}=\int_{0}^{2\pi} \dfrac{d^2}{dt^2} \big\{\phi (e^{it})\big\}~ \eta(e^{it},U_0,U)~ \frac{dt}{2\pi}
\end{equation}
holds, for every complex polynomial
$\phi(z)=\sum\limits_{k=0}^na_kz^k;~n\geq 0;~a_k,z\in \mathbb{C}.$
In this connection, Gesztesy et al. posed an open question in  \cite[Open Question 11.2]{GePu} which says the following:
\begin{equation}\label{question}
	\text{\emph{Is the above distribution}}\quad  \eta \quad \text{\emph{in}} \quad \eqref{linshif} \quad \text{\emph{an}}\quad  L^1(\mathbb{T})-\text{\emph{function}} ?
\end{equation}
In 2012, Potapov and Sukochev provided an affirmative answer to the question mentioned above in \cite{PoSu}. In fact, they prove the following interesting theorem in \cite{PoSu}.

\begin{thm}(see \cite[Theorem 1]{PoSu})\label{PSthm}
	Let $U$ and $U_0$ be two contractions in an infinite dimensional separable Hilbert space $\hil$ such that $V:=U-U_0\in\hils$. Denote $U_s=U_0+sV, ~s\in[0,1]$. Then for any complex polynomial  $p(\cdot)$,
	\begin{align*}
		\Bigg\{p(U)-p(U_0)-\at{\dds}{s=0}\big\{p(U_s)\big\}\Bigg\}\in\boh
	\end{align*} 
	and there exists an $L^1(\mathbb{T})$-function $\eta$  (unique up to an analytic term) such that
	\begin{align}\label{eqin20}
		\operatorname{Tr}\Bigg\{p(U)-p(U_0)-\at{\dds}{s=0} \big\{p(U_s)\big\} \Bigg\}=\int_{\cir}
		p''(z)\eta(z)dz.
	\end{align}
	Moreover, for every given $\epsilon >0$, we can choose the function $\eta$ satisfying \eqref{eqin20} in such a way so that 
	\begin{equation}\label{eqin40}
		\|\eta\|_{L^1(\mathbb{T})}\leq \left(1+\epsilon\right) ~\|V\|_2^2. 
	\end{equation}
\end{thm}
\noindent Note that, for a description of a wider class of functions for which formulae \eqref{intequ3} and \eqref{intequ4} hold we
refer to \cite{Pe05} and for a higher order version of \eqref{intequ4} and \eqref{eqin20} to \cite{AP2011, PoSkSu, PoSkSu16, Sk, ST}. 

The following are the main contributions in this article.

\begin{itemize}
	\item First of all, we supply a new proof of the above Theorem~\ref{PSthm} whenever $U_0$ is a normal contraction and $U$ is a contraction such that $U-U_0\in \mathcal{B}_2(\mathcal{H})$ (see Theorem~\ref{th4} and \ref{th5}), we believe for the first time, using the idea of finite-dimensional approximation method as in the works of Voiculescu, Sinha and Mohapatra, referred to earlier which in particular provides an affirmative answer to the above question \eqref{question}.
	\vspace{0.1in}
	
	\item Consequently, using the Sch\"{a}ffer matrix unitary dilation we also prove Theorem~\ref{PSthm} corresponding to a class of pairs of contractions $(T,T_0)$ such that $T-T_0\in \mathcal{B}_2(\mathcal{H})$ (see Theorem~\ref{th8} and Theorem~\ref{th7}).
	\vspace{0.1in}
	
	\item Next, by using our main theorem and using the Cayley transform of self-adjoint operators, we obtain the Koplienko trace formula corresponding to a pair of self-adjoint operators $(H, H_0)$ with the same domain in $\mathcal{H}$ and under the assumption that $H-H_0\in \mathcal{B}_2(\mathcal{H})$ (see Theorem~\ref{selfth}).
	\vspace{0.1in}
	
	\item  Moreover, by using Theorem~\ref{PSthm}, we prove the Koplienko trace formula for a pair of maximal dissipative operators $(L,L_0)$ under the assumption $(L+i)^{-1}-(L_0+i)^{-1}\in \mathcal{B}_2(\mathcal{H})$ (see Theorem~\ref{dissipthm}) for the first time.
	\vspace{0.1in}
	
	\item At the end, using the idea of  finite-dimensional approximation method, we have extended the Koplienko-Neidhardt trace formula for a class of pairs of contractions $(T,T_0)$ via multiplicative path (see Theorem~\ref{main}) for the first time.
\end{itemize}	

The significant differences between our method and the method applied in \cite{GePu, PoSu} are the following. 

\begin{itemize}
	\item Our approach in this article is different from that of \cite{GePu, PoSu} and is probably closer to Koplienko and Neidhardt's original approach (see \cite{Ko, NH}). In \cite{PoSu}, the authors proved Krein type formula (see \cite[Theorem 6]{PoSu}) to obtain Theorem~\ref{PSthm} by approximating the perturbation operator (and not the initial operator) via trace class operators but still, they were in an infinite-dimensional setting to deal with the problem which makes a major contrast in comparison to our context. In other words, in our setting, we reduce the problem into a finite-dimensional one by truncating both the initial operator and the perturbation operator simultaneously via finite-dimensional projections $\{P_n\}$ (see Theorem~\ref{th3(1)} and \ref{th3(2)}).
	\vspace{0.1in}
	
	\item Moreover, in our setting, we calculate the shift function explicitly by performing integration by-parts and using semi-spectral measures for contractions (see Theorem~\ref{th2(1)} and \ref{th2(2)}), and it is one of the significant steps in our context to get the shift function in an infinite-dimensional case which is not the principle essence in the approach mentioned in \cite{GePu, PoSu}.
	\vspace{0.1in}
	
	\item Furthermore, using our approach we obtain a slightly better upper bound of the $L^1(\mathbb{T})$-norm of $\eta$ (see Theorem~\ref{th4} and Theorem~\ref{th5}) compared to \eqref{eqin40} in Theorem~\ref{PSthm}. 
	
\end{itemize}
\vspace{0.1in}

The rest of the paper is organized as follows: Section 2 deals with some well-known concepts/results that are essential in the later sections. In Section 3, we give the proof of the Koplienko trace formula and Koplienko-Neidhardt trace formula for pairs of contractions when $dim~\mathcal{H}<\infty$. Section 4 is devoted to reducing the problem into finite dimensions, and in Section 5, we prove the required trace formula by appropriate limiting argument. Consequently, in Section 6, we prove the trace formula for a class of pairs of contractions. Section 7 and Section 8 deal with the trace formula
for a pair of self-adjoint operators and maximal dissipative operators respectively. At the end, in Section 9, we prove the Koplienko-Neidhardt trace for formula for a class of pairs of contractions via multiplicative path.

\section{\bf Preliminaries}
Recall that a \emph{semi-spectral measure} $\mathcal{E}$ on a measurable space $(\mathcal{X},\mathcal{B})$ is a map on the $\sigma$-algebra $\mathcal{B}$ with values in the set of bounded linear operators on a Hilbert space $\mathcal{H}$ that is countably additive in the strong operator topology and such that
\begin{equation*}
	\mathcal{E}(\Delta) \geq 0\text{ for all } \Delta\in \mathcal{B},~ \mathcal{E}(\emptyset)=0,\text{ and } \mathcal{E}(\mathcal{X})=I.
\end{equation*}
It is interesting to observe that by Naimark's theorem \cite{Nai} each semi-spectral measure $\mathcal{E}$ has a \emph{spectral dilation}, that is a
spectral measure $E$ on the same measurable space  $(\mathcal{X},\mathcal{B})$ that takes values in the set
of orthogonal projections on a Hilbert space $\mathcal{K}$ containing $\mathcal{H}$, and such that
\begin{equation*}
	\mathcal{E}(\Delta)=P_{\mathcal{H}}E(\Delta)|_{\mathcal{H}}, \quad \Delta\in \mathcal{B},
\end{equation*}
where $P_{\mathcal{H}}$ is the orthogonal projection on $\mathcal{K}$ onto $\mathcal{H}$. Integrals with respect to semi-spectral measures are defined in the following way:
\begin{equation}\label{semieq}
	\int_{\mathcal{X}}\phi(x)~\mathcal{E}(dx)=P_{\mathcal{H}}\left(\int_{\mathcal{X}}\phi(x)~E(dx)\right)\Big|_{\mathcal{H}},\quad \phi\in C(\mathcal{X}),
\end{equation}
where $C(\mathcal{X})$ is the space of all continuous functions on $\mathcal{X}$.
Recall that each contraction $T$ (that is, $\|T\|\leq 1$) on a Hilbert space $\mathcal{H}$ has a minimal unitary dilation $U$, that is $U$ is a unitary operator on a Hilbert space $\mathcal{K}$,~ $\mathcal{H}\subseteq \mathcal{K}$,~$T^n=P_{\mathcal{H}}U^n|_{\mathcal{H}}$ for $n\geq 0$ and $\mathcal{K}$ is the closed linear span of $U^n\mathcal{H}$,~$n\in \mathbb{Z}$ (see \cite{NGFO}, Ch. I, Theorem~4.2). Here $P_{\mathcal{H}}$ is the orthogonal projection on $\mathcal{K}$ onto $\mathcal{H}$. The \emph{semi-spectral measure} $\mathcal{E}_T$ of $T$ is defined by
\begin{equation}\label{spectralmeasurecon}
	\mathcal{E}_T(\Delta)\stackrel{\text{def}}{=}P_{\mathcal{H}}E_U(\Delta)|_{\mathcal{H}},
\end{equation}
where $\int\limits_0^{2\pi} e^{it}~E_U(dt)$ is the spectral representation of $U$, $E_U(\cdot)$ is the spectral measure determined uniquely by the unitary operator $U$ such that it is continuous at $t=0$, that is, $E_U(0)=0$ (see page 281, \cite{RN}), and $\Delta$ is a Borel subset of $[0,2\pi]$. Then it is easy to see that
\begin{equation}\label{funcalcon}
	T^n=\int_0^{2\pi}e^{int}~\mathcal{E}_T(dt),\quad n\in \mathbb{N}\cup \{0\}.
\end{equation}
It is important to note that, for the first time, semi-spectral measures in perturbation theory (in particular, in the context of double operator integrals) were used in \cite{P1987}. In this connection, it is important to note that the existence of integrals with respect to a semi-spectral measure $\mathcal{E}_T(\cdot)$ appeared in the above formula \eqref{semieq} as well as direct functional calculus (without using the dilation reasoning) is discussed in detail in \cite{MM2003}. Moreover, a concept of multiplicity function for a semi-spectral measure is discussed there and it was established for the first time in \cite{MM2003} that the measures $\mathcal{E}_T(\cdot)$ and $E_U(\cdot)$ are spectrally equivalent. Note also that a simple proof of the classical  Naimark's theorem \cite{Nai} was obtained in \cite{MM2003}. For more on semi-spectral measures and related stuff, we refer to \cite{ALPP}. 

For $1\leq p\leq \infty$, the Hardy space $H^p(\mathbb{T})$ stand for the set $\{f\in L^p(\mathbb{T}):~ \hat{f}(n)=0,~\text{for all}~ n<0\}$. One of the basic facts about $H^p(\mathbb{T})$ spaces is that they have preduals (see \cite[Theorem~4.15]{hbspace}). In particular, $H^{\infty}(\mathbb{T})$ is isometrically isomorphic to the dual of the factor-space $L^1(\mathbb{T})/H^1(\mathbb{T})$ and furthermore, for every $f\in L^1(\mathbb{T})$, we have 
\begin{equation*}
	\|[f]\|_{L^1(\mathbb{T})/H^1(\mathbb{T})} = \sup_{\substack{\|g\|_{H^{\infty}(\mathbb{T})}\leq 1 }}\Bigg\lvert\int_{\cir} g(z) f(z)dz\Bigg\lvert.
\end{equation*}
Now we denote the set of all complex polynomials by $\mathcal{P}(\cir)$. It is important to note that the supremum in the above equality can be taken over the set $\mathcal{P}(\cir)$. In other words, we need the following well known result to calculate the norm on the factor space $L^1(\mathbb{T})/H^1(\mathbb{T})$ in later sections.
\begin{lem}\label{supnorm}
	For every $f\in L^1(\mathbb{T})$, the equality
	\begin{equation*}
		\begin{split}
			\|[f]\|_{L^1(\mathbb{T})/H^1(\mathbb{T})} = \sup_{\substack{g\in\mathcal{P}(\cir);~\|g\|_{H^{\infty}}\leq 1 }}\Bigg\lvert\int_{\cir} g(z) f (z)dz\Bigg\lvert 
		\end{split}
	\end{equation*}
	holds, where $\mathcal{P}(\cir)$ is the set of all complex polynomials.
\end{lem}
Moreover, to prove our main results, we need the following fundamental estimate, which is obtained in \cite[Theorem~6.1]{KISHUL} (see also \cite[Theorem~4.2 and (3.2)]{Peller09}).

\begin{thm}\label{estimatethm}
	If $f\in \mathcal{P}(\cir)$, then for all contractions $T,T_0$ on $\hil$,
	\begin{equation}
		\begin{split}
			\left\|f(T)-f(T_0)\right\|_2\,&\leq\, \|f'\|_{\infty}~\|T-T_0\|_2\qquad\text{ if } T-T_0\in\hils,\\
			\text{ and }\left\|f(T)X-Xf(T_0)\right\|_2\,&\leq\, \|f'\|_{\infty}~\|TX-XT_0\|_2\qquad\text{ if }X\in\hils,
		\end{split}
	\end{equation} where $\|f'\|_{\infty} =\sup\limits_{t\in [0,2\pi)}|f'(e^{it})|$.
\end{thm}

Next we introduce the set of functions $\mathcal{F}_2(\mathbb{T})$ and $\mathcal{F}_2^+(\mathbb{T})$ which will be useful in our later sections.
\begin{equation*}
	\begin{split}
		\mathcal{F}_2(\mathbb{T}):=&\Big\{f(z)=\sum_{k=-\infty}^{\infty}\hat{f}(k)z^k\in C^2(\mathbb{T}):\sum_{k=-\infty}^{\infty}|k|^2|\hat{f}(k)|<\infty \Big\},\\
		\mathcal{F}_2^+(\mathbb{T}):=&\Big\{f(z)=\sum_{k=0}^{\infty}\hat{f}(k)z^k\in C^2(\mathbb{T}):\sum_{k=0}^{\infty}|k|^2|\hat{f}(k)|<\infty \Big\},
	\end{split}
\end{equation*}
where $\big\{\hat{f}(k):~k\in \mathbb{Z}\big\}$ is the Fourier coefficients of $f$ and $C^2(\mathbb{T})$ is the collection of all $2$-times continuously differentiable functions on $\mathbb{T}$. Let $\phi\in\mathcal{F}_2(\mathbb{T})$ be such that $\phi(e^{it})=\sum\limits_{k=-\infty}^{\infty}\hat{\phi}(k)e^{ikt}$. Now we introduce the functions, namely  $\phi_{+}(e^{it})=\sum\limits_{k=0}^{\infty}\hat{\phi}(k)e^{ikt}$ and $\phi_{-}(e^{it})=\sum\limits_{k=1}^{\infty}\hat{\phi}(-k)e^{ikt}$. Then $\phi(e^{it})=\phi_{+}(e^{it})+\phi_{-}(e^{-it})$ and $\phi_{\pm}\in \mathcal{F}_2^+(\mathbb{T})$. Thus for a given contraction $T$ on $\mathcal{H}$, we set
\begin{align}\label{INTeq1}
	\phi_{+}(T)=\sum\limits_{k=0}^{\infty}\hat{\phi}(k)T^k,~   \phi_{-}(T)=\sum\limits_{k=1}^{\infty}\hat{\phi}(-k){T^*}^k, \text{ and } \phi(T)=\phi_{+}(T)+\phi_{-}(T).
\end{align}

\section{\bf Trace formula in finite dimension}
We begin the section with the following differentiation formula for monomials of contractions that can be established directly by definition of the G\^{a}teaux derivative (with convergence in the operator norm).	

\begin{lem}\label{th1}
	Let $T$ and $T_0$ be two contractions in an infinite dimensional separable Hilbert space $\hil$. Let $V=T-T_0$, $T_s=T_0+sV,~~s\in [0,1]$, and $p(z)=z^r~(r\geq 2), z\in \mathbb{T}$. Then 
	\begin{align}\label{exp1}
		\dds \big\{p(T_s)\big\}=\sum_{j=0}^{r-1} T_s^{r-j-1} V T_s^j. 		 
	\end{align}
\end{lem}

\begin{proof}
	For $p(z)=z^r~(r\geq 2), ~z\in \mathbb{T}$, we have
	\begin{align}\label{exp2}
		\frac{p(T_{s+h})-p(T_s)}{h}=\dfrac{1}{h} \sum_{j=0}^{r-1} T_{s+h}^{r-j-1}(T_{s+h}-T_s)T_s^j
		=\sum_{j=0}^{r-1} T_{s+h}^{r-j-1}VT^j_s,
	\end{align}
	and hence 
	\begin{align*}
		&\left\|\frac{p(T_{s+h})-p(T_s)}{h} - \sum_{j=0}^{r-1} T_s^{r-j-1} V T_s^j 	\right\|\\
		\leq&\, |h|\left\{ \sum_{j=0}^{r-2}\sum_{k=0}^{r-j-2} \left\|T_s+hV\right\|^{r-j-k-2}\|V\|\left\|T_s\right\|^k\|V\|\left\|T_s\right\|^j\right\},
	\end{align*}
	which converges to $0$ as $h\longrightarrow 0$. This completes the proof. 
\end{proof}
The following lemma essentially obtained in \cite[Theorem 2.2 (i)]{CDP} whenever $A\in \mathcal{B}_2(\mathcal{H})$ is self-adjoint and $T_0$ is a unitary operator by using the definition of the G\^{a}teaux derivative. But if we consider $A,T_0\in \mathcal{B}(\mathcal{H})$, then the similar proof is also valid for the pair $(A,T_0)$.  
\begin{lem}\label{multifinitelemma}
	Let $(A,T_0)$ be a pair of bounded linear operator in an infinite dimensional separable Hilbert space $\hil$. Let  $T_s=e^{isA}T_0,~~s\in \mathbb{R}$, and $p(z)=z^r~(r\in\mathbb{Z}), z\in \mathbb{T}$. Then 
	\begin{align}\label{conpathexp}
		\dds \big\{p(T_s)\big\}=
		\begin{cases}
			\sum\limits_{j=0}^{r-1} T_s^{r-j-1} (iA) T_s^{j+1} &\text{ if } r\geq 1\\
			\quad 0 &\text{ if } r=0\\
			-\sum\limits_{j=0}^{|r|-1} {(T_s^*)}^{|r|-j} (iA^*) {(T_s^*)}^{j} &\text{ if } r\leq -1
		\end{cases} 		 
	\end{align}
\end{lem}

The following theorem states Koplienko trace formula for pairs of contractions via linear path in finite dimension.
\begin{thm}\label{th2(1)}
	Let $(N,N_0)$ be a pair of contraction on a finite dimensional Hilbert space $\hil$, and $V=N-N_0$. Let $N_s=N_0+sV,~s\in[0,1]$ and $p(\cdot)$ be any complex polynomial. Then there exists a $L^1(\cir)$-function $\eta$ such that 
	\begin{align}\label{Equation1}
		\operatorname{Tr}\Bigg\{p(N)-p(N_0)-\at{\dds}{s=0} \big\{p(N_s)\big\} \Bigg\}=\int_{\cir}
		p''(z)\eta(z)dz,
	\end{align}
	where $p(\cdot)$ is any complex polynomial and 
	\begin{equation}\label{aeq1}
		\eta(z)=\int_{0}^{1}\operatorname{Tr} \Big[V\Big\{\mathcal{E}_0(Arg(z))-\mathcal{E}_s(Arg(z))\Big\}\Big]ds, ~z\in\cir,
	\end{equation}
	where $\mathcal{E}_s(\cdot)$ and $\mathcal{E}_0(\cdot)$ are the semi-spectral measures corresponding to the contractions $N_s$ and $N_0$ respectively and  $Arg(z)$ is the  principle argument of $z$. 
	Furthermore, the class of all $\eta$'s satisfying \eqref{Equation1} corresponds to a unique element $[\eta]\in L^1(\mathbb{T})/H^1(\mathbb{T})$ such that 
	\begin{equation}\label{eq35class}
		\|[\eta]\|_{L^1(\mathbb{T})/H^1(\mathbb{T})}
		\leq \frac{1}{2} ~\|V\|_2^2.
	\end{equation}
	
\end{thm}
\begin{proof}
	It will be sufficient to prove the theorem for $p(z)=z^r~, z\in \mathbb{T}$. Note that for $r = 0$ or $1$, both sides of \eqref{eq21} are identically zero. By using the cyclicity of trace, applying Lemma~\ref{th1}, and noting that the trace now is a finite sum, we have that for
	$p(z)=z^r~(r\geq 2), z\in \mathbb{T}$,
	\begin{align*}
		&\operatorname{Tr}\Big\{p(N)-p(N_0)-\at{\dds}{s=0} \big\{p(N_s)\big\}\Big\}=
		\operatorname{Tr}\Big\{\int_{0}^{1}\dds \big\{p(N_s)\big\}~ds\Big\}-\operatorname{Tr}\Big\{\at{\dds}{s=0} \big\{p(N_s)\big\}\Big\}\\
		& = \int_{0}^{1} \operatorname{Tr}\Big\{\sum_{j=0}^{r-1} N_s^{r-j-1} V N_s^j \Big\}~ds - \int_{0}^{1}\operatorname{Tr}\Big\{\sum_{j=0}^{r-1} N_0^{r-j-1} V N_0^j\Big\}~ds\\
		& = r\int_{0}^{1} \Big[\operatorname{Tr}\Big\{V\left(N_s^{r-1}-N_0^{r-1}\right)\Big\}\Big]~ds = \operatorname{Tr}\Big\{r V \int_{0}^{1} ds \int_{0}^{2\pi} e^{i(r-1)t} \big[\mathcal{E}_s(dt)-\mathcal{E}_0(dt)\big]\Big\},
	\end{align*}
	where $\mathcal{E}_s(\cdot)$ and $\mathcal{E}_0(\cdot)$ are the semi-spectral measures corresponding to the contractions $N_s$ and $N_0$ respectively (see \eqref{spectralmeasurecon} and \eqref{funcalcon} in Section 2). Next by performing integration by-parts we have that 
	\begin{align*}
		& \operatorname{Tr}\Big\{p(N)-p(N_0)-\at{\dds}{s=0} \big\{p(N_s)\big\}\Big\} \\
		&= \operatorname{Tr}\Big\{ r V \int_{0}^{1} ds~  \Big(e^{i(r-1)t} \big[\mathcal{E}_s(t)-\mathcal{E}_0(t)\big]\Big|_{t=0}^{2\pi}-i(r-1) \int_{0}^{2\pi}e^{i(r-1)t} \big[\mathcal{E}_s(t)-\mathcal{E}_0(t)\big] dt\Big)~\Big\}\\
		&=i r(r-1)\int_{0}^{2\pi} e^{i(r-1)t} \Big\{\int_{0}^{1}\operatorname{Tr} [V\big\{\mathcal{E}_0(t)-\mathcal{E}_s(t)\big\}]ds\Big\}dt,
	\end{align*} 
	which by substituting $z=e^{it},~~t\in[0,2\pi]$  and $~dt=\dfrac{dz}{iz}$ yields
	\begin{align*}
		&\operatorname{Tr}\Big\{p(N)-p(N_0)-\at{\dds}{s=0} \big\{p(N_s)\big\}\Big\} \\
		& =\int_{\cir}r(r-1)z^{r-2}\Big\{\int_{0}^{1}\operatorname{Tr} [V\big\{\mathcal{E}_0(Arg(z))-\mathcal{E}_s(Arg(z))\big\}]~ds\Big\} ~dz
		=\int_{\cir}
		p''(z)\eta(z) dz,
	\end{align*}
	where $Arg(z)$ is the  principle argument of $z$ and we have set
	\begin{equation*}
		\eta(z)=\int_{0}^{1}\operatorname{Tr} \Big[V\Big\{\mathcal{E}_0(Arg(z))-\mathcal{E}_s(Arg(z))\Big\}\Big]ds, ~z\in\cir.
	\end{equation*}
	Let $f$ be a complex polynomial on $\cir$, and set $g(e^{it})=\int_{0}^{t}f(e^{is})ie^{is}ds,~t\in[0,2\pi]$. Next we observe that $g(e^{i2\pi})=g(e^{i0})=0$, and $\ddt \big\{g(e^{it})\big\}=ie^{it}f(e^{it})$. Now by using the above expression \eqref{aeq1} of $\eta$ and using Fubini's theorem to interchange the orders of integration and integrating by-parts, we have that 
	\begin{align*}
		&\int_{\cir} f(z)\eta(z) dz =\int_{0}^{2\pi} f(e^{it})\eta(e^{it}) ie^{it}dt =\int_{0}^{2\pi}\ddt \big\{g(e^{it})\big\}\eta(e^{it})dt\\
		&=\int_{0}^{2\pi}\ddt \big\{g(e^{it})\big\}\left(\int_{0}^{1}\operatorname{Tr} \Big[V\Big\{\mathcal{E}_{0}(t)-\mathcal{E}_{s}(t)\Big\}\Big]~ds\right)~dt\\
		&=\int_{0}^{1} ds \int_{0}^{2\pi}\ddt \big\{g(e^{it})\big\}~\operatorname{Tr} \Big[V\Big\{\mathcal{E}_{0}(t)-\mathcal{}E_{s}(t)\Big\}\Big]dt\\
		&=\int_{0}^{1} ds ~\Bigg(g(e^{it})~\operatorname{Tr} \Big[V\Big\{\mathcal{E}_{0}(t)-\mathcal{E}_{s}(t)\Big\}\Big]\Bigg\lvert_{t=0}^{2\pi} -\int_{0}^{2\pi}g(e^{it})~\operatorname{Tr} \Big[V\Big\{\mathcal{E}_{0}(dt)-\mathcal{E}_{s}(dt)\Big\}\Big]\Bigg)\\
		&=-\int_{0}^{1}ds\int_{0}^{2\pi}g(e^{it})~\operatorname{Tr} \Big[V\Big\{\mathcal{E}_{0}(dt)-\mathcal{E}_{s}(dt)\Big\}\Big] =\int_{0}^{1}ds~\operatorname{Tr} \Big[V\Big\{g(N_{s})-g(N_{0})\Big\}\Big].
	\end{align*}
	Therefore using Theorem~\ref{estimatethm} we get
	\begin{align*}
		& \Bigg\lvert\int_{\cir} f(z)\eta (z)dz\Bigg\lvert=~ \Bigg\lvert\int_{0}^{1}ds~\operatorname{Tr}\Big[V\Big\{g(N_{s})-g(N_{0})\Big\}\Big]\Bigg\lvert
		\leq \int_{0}^{1}ds~\Big\lvert \operatorname{Tr}\Big[V\Big\{g(N_{s})-g(N_{0})\Big\}\Big]\Big\lvert\\
		& \leq~\int_{0}^{1}\norm{V}_2\norm{g(N_{s})-g(N_{0})}_2
		\leq \int_{0}^{1}\norm{g'}_\infty\norm{V}_2\norm{N_{s}-N_{0}}_2\\
		&\leq \norm{f}_\infty\norm{V}_2^2\int_{0}^{1}s~ds
		~=\frac{1}{2}\norm{f}_\infty\norm{V}_2^2,
	\end{align*}
	and hence by using Lemma~\ref{supnorm} we conclude that 
	\begin{equation*}
		\begin{split}
			\|[\eta]\|_{L^1(\mathbb{T})/H^1(\mathbb{T})} = \sup_{\substack{f\in\mathcal{P}(\cir);~\|f\|_{H^{\infty}(\mathbb{T})}\leq 1 }}\Bigg\lvert\int_{\cir} f(z)\eta (z)dz\Bigg\lvert \leq \frac{1}{2}~\norm{V}_2^2.
		\end{split}
	\end{equation*}
	This completes the proof.
	
\end{proof}

The following theorem states Koplienko-Neidhardt trace formula for pairs of contractions via multiplicative path in finite dimension.
\begin{thm}\label{th2(2)}
	Let $T_0$ be a contraction in a finite dimensional Hilbert space $\hil$ and let $A=A^*$ $\in\mathcal{B}(\mathcal{H})$. Denote $T_s=e^{isA}T_0, ~s\in[0,1],$ and $T=T_1$. Then there exists an $L^1([0,2\pi])$-function $\tilde{\eta}$ such that   
	\begin{align}\label{eq21}
		\operatorname{Tr}\Bigg\{p(T)-p(T_0)-\at{\dds}{s=0} \big\{p(T_s)\big\} \Bigg\}=\int_{0}^{2\pi}
		\dfrac{d^2}{dt^2}\Big\{p(e^{it})\Big\}\tilde{\eta}(t)dt,
	\end{align}
	where $p(\cdot)$ is any complex polynomial on $\mathbb{T}$ with complex coefficients and 
	\begin{equation}\label{emul}
		\tilde{\eta}(t)=\int_{0}^{1}\operatorname{Tr} \Big[A\Big\{\mathcal{F}_0(t)-\mathcal{F}_s(t)\Big\}\Big]ds, ~t\in[0,2\pi],
	\end{equation}
	where $\mathcal{F}_s(\cdot)$ and $\mathcal{F}_0(\cdot)$ are the semi-spectral measures corresponding to the contractions $T_s$ and $T_0$ respectively. 
	Moreover, 	\begin{equation}\label{etabdd}
		\|[\tilde{\eta}]\|_{L^1(\mathbb{T})/H^1(\mathbb{T})}
		\leq \frac{1}{2} ~\|A\|_2^2.
	\end{equation}
	
\end{thm}
\begin{proof}
	It will be sufficient to prove the theorem for $p(z)=z^r~,r\in\mathbb{N}\cup \{0\}, z\in \mathbb{T}$. Note that for $r = 0$,  both sides of \eqref{eq21} are identically zero. Let $\mathcal{F}_s(\cdot)$ and $\mathcal{F}_0(\cdot)$ are the semi-spectral measures corresponding to the contractions $T_s$ and $T_0$ respectively (see \eqref{spectralmeasurecon} and \eqref{funcalcon} in Section 2). By using the cyclicity of trace, applying Lemma~\ref{multifinitelemma}, and noting that the trace now is a finite sum, we have that for
	$p(z)=z^r~(r\geq 2), z\in \mathbb{T}$,
	\begin{align*}
		& \operatorname{Tr}\Big\{p(T_1)-p(T_0)-\at{\dds p(T_s)}{s=0}\Big\}= \operatorname{Tr}\left\{\int_{0}^{1}\left(\dds p(T_s)-\at{\ddt p(T_t)}{t=0}\right)ds\right\}\\
		&= \operatorname{Tr}\left\{\int_{0}^{1}\left(\sum\limits_{j=0}^{r-1} T_s^{r-j-1} (iA) T_s^{j+1}-\sum\limits_{j=0}^{r-1} T_0^{r-j-1} (iA) T_0^{j+1}\right)ds\right\}\\
		&= \operatorname{Tr}\left\{(ir)A\int_{0}^{1}\left(T_s^r-T_0^r\right)ds\right\}\\
		&= \operatorname{Tr}\left\{(ir)A\int_{0}^{1}ds\int_{0}^{2\pi}e^{irt}\Big(\mathcal{F}_s(dt)-\mathcal{F}_0(dt)\Big)\right\}\\
		&= \operatorname{Tr}\left[ (ir) A \int_{0}^{1} ds~  \Big\{e^{irt} \big(\mathcal{F}_s(t)-\mathcal{F}_0(t)\big)\Big|_{t=0}^{2\pi}-ir \int_{0}^{2\pi}e^{irt} \big(\mathcal{F}_s(t)-\mathcal{F}_0(t)\big) dt\Big\}~\right]\\
		&=\int_{0}^{2\pi}(ir)^2 e^{irt}
		\left[\int_0^1 \operatorname{Tr}\Big\{A\big(\mathcal{F}_0(t)-\mathcal{F}_s(t)\big)\Big\} ds\right]dt = \int_{0}^{2\pi} \frac{d^2}{dt^2}\Big\{p(e^{it})\Big\} \eta(t) dt.
	\end{align*}
	Therefore we have the formula \eqref{eq21} by setting $$\tilde{\eta}(t)=\int_0^1 \operatorname{Tr}\big\{A[\mathcal{F}_0(t)-\mathcal{F}_s(t)\big]\big\} ds. $$
	By repeating the similar argument as done in Theorem \ref{th2(1)} we obtain \eqref{etabdd}. This completes the proof.
\end{proof}

\begin{rem}
	It is easy to observe that the function $\tilde{\eta}$ in \eqref{emul} is real-valued. Also it is worth mentioning that, if we consider the polynomial $p(z)=z^r$ for $r\leq -1$, then by performing the similar calculations we also obtain the formula \eqref{eq21} along with the same spectral shift function $\tilde{\eta}$ as in \eqref{etabdd}. Therefore the identity \eqref{eq21} holds for every trigonometric polynomial and hence by the uniqueness of the Fourier series we conclude that the spectral shift function $\tilde{\eta}(\cdot)$ is unique up to an additive constant. Moreover, the formula \eqref{eq21} can also be extended for the class $\mathcal{F}_2(\cir)$.
\end{rem}

\section{\bf Reduction to finite dimension}
We begin the section by stating (without proof) the approximation theorem  which is essential in our context to reduce the problem into a finite dimensional one. Note that the following theorem is a special case of Theorem 2.2 in \cite{KBSAC}. Moreover, it is also worth mentioning that Voiculescu \cite{Voiapp} had earlier obtained related (though not the same) results. 

\begin{thm}(See \cite[Theorem 2.2]{KBSAC})\label{aproxth1}
	Let $(A_1, A_2)$ be a pair of commuting  bounded self-adjoint operators in an 
	infinite-dimensional separable Hilbert space $\mathcal{H}$.
	Then there exists a sequence $\{P_k\}$ of finite-rank projections such that, for $i=1,2$ 
	\begin{align}\label{eqapprox}
		P_k \uparrow I, ~~~\text{ and }~~ \left\|[A_i,P_k]\right\|_2 \longrightarrow 0, \text{ as } k\longrightarrow\infty.
	\end{align}	
\end{thm}
Using the above theorem we have the following useful lemmas which will be useful to reduce the problem into a finite dimensional case.
\begin{lem}\label{approxlm1}
	Let $(H_1, H_2)$ be a pair of commuting  bounded self-adjoint operators in an 
	infinite-dimensional separable Hilbert space $\mathcal{H}$. Let $A\in \mathcal{B}_2(\mathcal{H})$ be a self-adjoint operator and let $B\in \mathcal{B}_2(\mathcal{H})$. Then for $i=1,2$, there exists a sequence $\{P_k\}$ of finite rank projections  such that $P_k \uparrow I$, and 
	\begin{align*}
		\left\|P_k^\perp H_i P_k\right\|_2,~~\left\|P_k^\perp A P_k\right\|_2,~~\left\|P_k^\perp B P_k\right\|_2, \left\|P_k^\perp B^* P_k\right\|_2\longrightarrow 0 \text{ as } k\longrightarrow\infty.
	\end{align*}
\end{lem}
\begin{proof}
	Now by applying Theorem \ref{aproxth1} corresponding to the pair $(H_1,H_2)$, there exists a sequence  $\{P_k\}$ finite rank projections such that $P_k\uparrow I$, and
	\begin{align*}
		\left\|P_k^\perp H_i P_k\right\|_2\longrightarrow 0 \text{ as } k\longrightarrow \infty \quad \text{for}\quad i=1,2.
	\end{align*}
	Let  $\sum\limits_{j=1}^{\infty}\lambda_j\langle \cdot, e_j \rangle f_j$ and $\sum\limits_{j=1}^{\infty}\mu_j\langle \cdot, g_j \rangle h_j$ be the corresponding canonical decomposition of $A$ and $B$ respectively, where  $\sum\limits_{j=1}^{\infty}\lambda_j^2<\infty$,~$\sum\limits_{j=1}^{\infty}\mu_j^2<\infty$ and  $\{e_j\}$, $\{f_j\}$, $\{g_j\}$ and $\{h_j\}$ are set of orthonormal vectors. Now given $n\in\mathbb{N}$, there exists $L_n\in\mathbb{N}$ such that 
	\begin{align*}
		\sum\limits_{j=L_n+1}^{\infty}\lambda_j^2<\dfrac{1}{n}, \text{ and } \sum\limits_{j=L_n+1}^{\infty}\mu_j^2<\dfrac{1}{n}.
	\end{align*} 
	Set $A_{L_n}=\sum\limits_{j=1}^{L_n}\lambda_j\langle \cdot, e_j \rangle f_j$, 
	$A_{L_n}^*=\sum\limits_{j=1}^{L_n}\lambda_j\langle \cdot, f_j \rangle e_j$, $B_{L_n}=\sum\limits_{j=1}^{L_n}\mu_j\langle \cdot, g_j \rangle h_j$, $B_{L_n}^*=\sum\limits_{j=1}^{L_n}\mu_j\langle \cdot, h_j \rangle g_j$, and $\epsilon_n=\min\left\{\dfrac{1}{n},~\dfrac{\dfrac{1}{n}}{\sum\limits_{j=1}^{L_n}\lambda_j},~ \dfrac{\dfrac{1}{n}}{\sum\limits_{j=1}^{L_n}\mu_j}\right\}$. Since $P_k\uparrow I$ as $k\to \infty$, then there exists a natural number $a_n$ such that for each $f\in \{e_1,e_2,\ldots,e_{L_n}\}\cup\{f_1,f_2,\ldots,f_{L_n}\}\cup\{g_1,g_2,\ldots,g_{L_n}\}\cup\{h_1,h_2,\ldots,h_{L_n}\}$ we have
	$$\|(I-P_k)f\|<\epsilon_n \quad\forall~ k\geq a_n.$$
	Next we choose $a_n\in \mathbb{N}$ such that $a_n<a_{n+1}$ for each $n\in\mathbb{N}$. Therefore corresponding to the sub-sequence $\{P_{a_n}\}$, we have for $i=1,2,$
	\begin{align*}
		\left\|P_{a_n}^\perp H_i P_{a_n}\right\|_2 &\longrightarrow 0,\\
		\left\|P_{a_n}^\perp A P_{a_n}\right\|_2\leq &\left\|P_{a_n}^\perp (A-A_{L_n}) P_{a_n}\right\|_2+\left\|P_{a_n}^\perp A_{L_n} P_{a_n}\right\|_2\\
		\leq &\left\|A-A_{L_n}\right\|_2+ \left\|P_{L_n}^\perp A_{L_n}\right\|_2
		< \dfrac{1}{n} + \epsilon_n ~\left(\sum\limits_{j=1}^{L_n}\lambda_j^2\right)^{\frac{1}{2}}\leq \dfrac{2}{n} \longrightarrow 0 \quad	\text{as} \quad n\longrightarrow \infty.
	\end{align*}
	Similarly we have
	\begin{align*}
		&\left\|P_{a_n}^\perp A^* P_{a_n}\right\|_2< \dfrac{1}{n} + \epsilon_n ~\left(\sum\limits_{j=1}^{L_n}\lambda_j^2\right)^{\frac{1}{2}}\leq \dfrac{2}{n} \longrightarrow 0,\\
		& \left\|P_{a_n}^\perp B P_{a_n}\right\|_2< \dfrac{1}{n} + \epsilon_n ~\left(\sum\limits_{j=1}^{L_n}\mu_j^2\right)^{\frac{1}{2}}\leq \dfrac{2}{n} \longrightarrow 0, \\
		& \text{and}\quad \left\|P_{a_n}^\perp B^* P_{a_n}\right\|_2< \dfrac{1}{n} + \epsilon_n ~\left(\sum\limits_{j=1}^{L_n}\mu_j^2\right)^{\frac{1}{2}}\leq \dfrac{2}{n} \longrightarrow 0 \quad \text{as} \quad n\longrightarrow \infty.
	\end{align*}
	This complete the proof.
\end{proof}
\begin{lem}\label{appthm}
	Let $N_0$ be a normal contraction on a separable Hilbert space $\hil$, and let $V\in\hils$. Let $T_0=N_0+V$ and $A=A^*\in\hils$. Set $T_s=e^{isA}T_0,~s\in[0,1]$ and $T=T_1$. Then there exists a sequence $\{P_n\}$ of finite rank projections  such that for every $k\in \mathbb{N}$, each of the following terms
	\begin{align*}
		&(i)~~\|P_n^\perp N_0P_n\|_2,~~(ii)~~\|P_n^\perp V\|_2,~~(iii)~~\|P_n^\perp V^*\|_2,~~(iv)~~\left\|\left( T^{k}-T_{n}^{k} \right)P_n\right\|_2,\\
		&(v)~~\left\|\left( T_{0}^{k}-T_{0,n}^{k} \right)P_n\right\|_2,~~(vi)~~\left\|P_n^\perp\left( e^{iA}-I \right)\right\|_2,(vii)~~\left\|P_n\left( e^{iA}-e^{iA_n}\right)\right\|_2,\\
		&(viii)~~\left\|\left(e^{iA}-iA-e^{iA_n}+iA_n\right)P_n\right\|_1,\quad \text{and} \quad (ix)~~\left\|\left(e^{iA}-iA-I\right)P_n^\perp\right\|_1
	\end{align*}
	converges to zero as $n\longrightarrow \infty$, where $A_n=P_nAP_n$, $T_{0,n} =P_nT_0P_n$,  $T_n=e^{iA_n}T_{0,n}$ and $T_{s,n}=e^{isA_n}T_{0,n}$.
\end{lem}
\begin{proof}
	Since $N_0$ is a normal contraction, then $N_0+N_0^*$ and $N_0-N_0^*$ are two commuting self-adjoint operators on $\mathcal{H}$. Therefore by applying Lemma \ref{approxlm1} corresponding to the pair $(N_0+N_0^*, N_0-N_0^*)$, there exists a sequence $\{P_n\}$ of finite rank projections such that 
	\begin{align*}
		\left\|P_n^\perp(N_0+N_0^*)P_n\right\|_2, \left\|P_n^\perp(N_0-N_0^*)P_n\right\|_2, \left\|P_n^\perp A\right\|_2, \left\|P_n^\perp V\right\|_2, \left\|P_n^\perp V^*\right\|_2\longrightarrow 0 \text{  as }n\longrightarrow \infty,
	\end{align*} which immediately implies that  
	\begin{align}\label{eqp}
		\left\|P_n^\perp N_0 P_n\right\|_2, \left\|P_n^\perp N_0^* P_n\right\|_2,\left\|P_n^\perp T_0 P_n\right\|_2, \left\|P_n^\perp T_0^* P_n\right\|_2 \longrightarrow 0 \text{  as }n\longrightarrow \infty.
	\end{align}
	This conclude $(i),~(ii)$ and $(iii)$. The proof of $(iv),~(v),~(vi),~(vii)$ and $(ix)$ can be obtained similarly by mimicking the proof of Lemma~3.4 given in \cite[Lemma~3.4]{CDP}. For $(viii)$, consider the following
	\begin{align*}
		&\left\|\left(e^{iA}-iA-e^{iA_n}+iA_n\right)P_n\right\|_1\\
		=&\left\|\sum_{k=2}^{\infty}\dfrac{1}{k!}~\sum_{j=0}^{k-1}\left\{(iA)^{k-j-1}(iA-iA_n)(iA_n)^j\right\}P_n\right\|_1\\
		=&\left\|\sum_{k=2}^{\infty}\dfrac{1}{k!}~\left[\sum_{j=0}^{k-1}\left\{(iA)^{k-j-1}(iA-iA_n)(iA_n)^j\right\}P_n\right]\right\|_1\\
		=&\Big\|\sum_{k=2}^{\infty}\dfrac{1}{k!}~\Big[(iA)^{k-1}i\left(P_n^\perp AP_n\right)+i\left(P_n^\perp AP_n\right)(iA_n)^{k-1}\\
		&+\sum_{j=1}^{k-2}\left\{(iA)^{k-j-1}i\left(P_n^\perp AP_n\right)(iA_n)^j\right\}P_n\Big]\Big\|_1\\
		=&\sum_{k=2}^{\infty}\dfrac{1}{k!}~\Big[\|A\|^{k-2}\|A\|_2\left\|P_n^\perp AP_n\right\|_2+\left\|P_n^\perp AP_n\right\|_2\|A_n\|^{k-2}\|A_n\|_2\\
		&\hspace{4cm}+\sum_{j=1}^{k-2}\left\{\|A\|^{k-j-1}\left\|P_n^\perp AP_n\right\|_2\|A_n\|^{j-1}\|A_n\|_2\right\}\Big]\\
		=&\left(\sum_{k=2}^{\infty}\dfrac{1}{(k-1)!}~ \|A\|^{k-2}\right)\|A\|_2\left\|P_n^\perp AP_n\right\|_2\\
		=&\|A\|^{-1}\left(e^{\|A\|}-1\right)\|A\|_2\left\|P_n^\perp AP_n\right\|_2 \longrightarrow 0 \text{ as } n\longrightarrow \infty.
	\end{align*}
	This completes the proof. 
\end{proof}
\begin{rem}\label{adj}
	Note that the expressions $(iv)$ and $(v)$ in Lemma~\ref{appthm} also converge to zero as $n\longrightarrow \infty$ for any $k\in \mathbb{Z}$, where we interpret $T^{-1}$ as the adjoint of $T$.
\end{rem}
The following two theorems show how the above Lemma~\ref{appthm} can be used to reduce the relevant problem into a
finite-dimensional one.
\begin{thm}\label{th3(1)}
	Let $N_0$ be a normal contraction on a separable Hilbert space $\hil$, and let $V\in\hils$. Let $N_s=N_0+sV,~s\in[0,1],~N=N_1$ and let $p(\cdot)$ be any complex polynomial. Then there exists a sequence $\{P_n\}$ of finite rank projections  such that 
	\begin{align}\label{eq18}
		\left\|\Bigg\{p(N)-p(N_0)-\at{\dds}{s=0}\big\{p(N_s)\big\} \Bigg\}-\Bigg\{P_n\left(p(N_n)-p(N_{0,n})-\at{\dds}{s=0}\big\{p(N_{s,n})\big\}\right) P_n \Bigg\}\right\|_1
	\end{align}
	$\longrightarrow 0$ as $n\longrightarrow \infty$, where $N_n=P_nNP_n,~N_{0,n}=P_nN_0P_n,$ and $N_{s,n}=P_nN_sP_n$.
\end{thm}
\begin{proof}
	It will be sufficient to prove the theorem for $p(z)=z^r,r\in\mathbb{N}, z\in\cir$. Note that for $r=0$ or $1$, the expressions inside the trace norm in \eqref{eq18} are identically zero. Let $r\geq 2$. Now by using the sequence $\{P_n\}$ of finite rank projections as obtained in Lemma~\ref{appthm}, and using an expression similar to \eqref{exp1} in $\mathcal{B}(\mathcal{H})$, we have that
	\begin{align}\label{eq19}
		\nonumber&~\left\|\Big(N^r-N_0^r- \sum_{j=0}^{r-1} N_{0}^{r-j-1}VN_0^j\Big)- P_n\Big(N_n^r-N_{0,n}^r- \sum_{j=0}^{r-1} N_{0,n}^{r-j-1}VN_{0,n}^j\Big)P_n\right\|_1\\
		\nonumber=&~\Bigg\|\sum_{\substack{\alpha+\beta=r-1\\\alpha\geq 1 ~\&~ \beta\geq 0}}\Big[ (N^{\alpha}-N_{0}^{\alpha})VN_0^\beta- P_n\Big( (N_{n}^{\alpha}-N_{0,n}^{\alpha})VN_{0,n}^\beta\Big)P_n\Big]\Bigg\|_1\\
		\nonumber=&~\Bigg\|\sum_{\substack{\alpha+\beta=r-1\\\alpha\geq 1 ~\&~ \beta\geq 0}}\Bigg[ \Big\{(N^{\alpha}-N_{n}^{\alpha})P_nVN_0^\beta + N^{\alpha}P_n^\perp VN_0^\beta+ N_n^{\alpha} VP_n (N_0^\beta-N_{0,n}^\beta) + N_n^{\alpha} VP_n^\perp N_0^\beta\Big\}\\
		\nonumber&~ -\Big\{ (N_0^{\alpha}-N_{0,n}^{\alpha})P_nVN_0^\beta + N_{0}^{\alpha}P_n^\perp VN_0^\beta+ N_{0,n}^{\alpha} VP_n (N_0^\beta-N_{0,n}^\beta) + N_{0,n}^{\alpha} VP_n^\perp N_0^\beta \Big\}\Bigg]\Bigg\|_1\\
		\nonumber \leq&~\sum_{\substack{\alpha+\beta=r-1\\\alpha\geq 1 ~\&~ \beta\geq 0}}  \Big\{ \|(N^{\alpha}-N_{n}^{\alpha})P_n\|_2 \|V\|_2 + \alpha\|V\|_2~ (\|P_n^\perp V\|_2+\|V P_n^\perp \|_2)\\
		&\hspace{.5in}+ 2\|V\|_2 \|P_n (N_0^\beta-N_{0,n}^\beta)\|_2
		+\|(N_0^{\alpha}-N_{0,n}^{\alpha})P_n\|_2 \|V\|_2\Big\}. 
	\end{align}
	Now using the estimates listed in  Lemma~\ref{appthm} along with the Remark~\ref{adj}, we conclude that the expression in the right hand side of \eqref{eq19} converges to zero as $n\longrightarrow \infty$. This completes the proof.
	
\end{proof}

\begin{thm}\label{th3(2)}
	Let $T_0=N_0+V$ be a contraction on a separable Hilbert space $\hil$, where $N_0$ is a bounded normal operator and $V\in\hils$. Let $A=A^*\in\hils$, $T_s=e^{isA}T_0, ~s\in[0,1],$~$T=T_1$ and $p(\cdot)$ be any trigonometric polynomial on $\cir$ with complex coefficients. Then there exists a sequence  $\{P_n\}$ of finite rank projections in $\mathcal{H}$ such that
	\begin{align}\label{eqapp}
		\left\|\left[\Big\{p(T)-p(T_0)-\at{\dds}{s=0}p(T_s) \Big\}-P_n\Big\{p(T_n)-p(T_{0,n})-\at{\dds}{s=0}p(T_{s,n}) \Big\}P_n\right]\right\|_1
	\end{align}
	$\longrightarrow 0 \text{ as } n\to\infty,$ where $A_n=P_nAP_n$, $T_{0,n} =P_nT_0P_n$,  $T_n=e^{iA_n}T_{0,n}$ and $T_{s,n}=e^{isA_n}T_{0,n}$.
\end{thm}
\begin{proof}
	It will be  sufficient to prove the theorem for $p(z)=z^r,r\in\mathbb{Z}, z\in\cir$. Note that for $r=0$, the expression inside the trace norm in \eqref{eqapp} is identically zero. For $r\geq 1$, using the sequence $\{P_n\}$ of finite rank projections as obtained in Lemma~\ref{appthm}, and using an expression similar to \eqref{conpathexp} in $\mathcal{B}(\mathcal{H})$, we have that
	\begin{align*}
		\nonumber&\left[\Big\{p(T)-p(T_0)-\at{\dds}{s=0}p(T_s)\Big\}- P_n\Big\{p(T_n)-p(T_{0,n})-\at{\dds}{s=0}p(T_{s,n})\Big\}P_n\right]\\
		\nonumber=&\left[\Big\{T^r-T_0^r- \sum_{j=0}^{r-1} T_{0}^{r-j-1}(iA)T_0^{j+1}\Big\}- P_n\Big\{T_n^r-T_{0,n}^r- \sum_{j=0}^{r-1} T_{0,n}^{r-j-1}(iA_n)T_{0,n}^{j+1}\Big\}P_n\right]\\
		\nonumber=&\sum_{\substack{\alpha+\beta=r\\\alpha\geq 0 ~\&~ \beta\geq 1}}\Bigg[ \Big\{T^{\alpha}(e^{iA}-I)T_0^{\beta}-T_{0}^{\alpha}(iA)T_0^{\beta}\Big\}
		-P_n \Big\{T_{n}^{\alpha}P_n(e^{iA_n}-I)P_nT_{0,n}^{\beta}- T_{0,n}^{\alpha}(iA_n)T_{0,n}^{\beta}\Big\}P_n\Bigg]\\
		\nonumber=&\sum_{\substack{\alpha+\beta=r\\\alpha\geq 0 ~\&~ \beta\geq 1}}\Bigg[ \Big\{\big(T^{\alpha}-T^{\alpha}_0\big)(e^{iA}-I)T_0^{\beta}+T_{0}^{\alpha}(e^{iA}-iA-I)T_0^{\beta}\Big\}\\
		\nonumber&\hspace{1in} -P_n \Big\{\left(T_{n}^{\alpha}-T_{0,n}^{\alpha}\right)P_n(e^{iA_n}-I)P_nT_{0,n}^{\beta}+ T_{0,n}^{\alpha}(e^{iA_n}-iA_n-I)T_{0,n}^{\beta}\Big\}P_n\Bigg]\\
		\nonumber=&\sum_{\substack{\alpha+\beta=r\\\alpha\geq 0 ~\&~ \beta\geq 1}}\Bigg[\Bigg\{\big( T^{\alpha}-T_{n}^{\alpha}- T_{0}^{\alpha}+T_{0,n}^{\alpha} \big)P_n(e^{iA}-I)T_0^{\beta}+ \big(T^{\alpha}-T_{0}^{\alpha}\big)P_n^\perp(e^{iA}-I)T_0^{\beta}\\
		\nonumber&\hspace{1cm}+\left(T_{n}^{\alpha}-T_{0,n}^{\alpha}\right)P_n\left( e^{iA}-e^{iA_n}\right)T_0^{\beta}+\left(T_{n}^{\alpha}-T_{0,n}^{\alpha}\right)P_n\left( e^{iA_n}-I\right)P_n\left(T_0^{\beta}-T_{0,n}^{\beta}\right)\Bigg\}\\
		\nonumber&-\Bigg\{(T_0^\alpha-T_{0,n}^{\alpha})P_n\left( e^{iA}-iA-I \right)T_0^{\beta}+ T_0^\alpha P_n^\perp \left( e^{iA}-iA-I \right)T_0^{\beta}\\
		&\hspace{1cm}+ T_{0,n}^\alpha P_n\left(e^{iA}-iA-e^{iA_n}+iA_n\right)T_0^{\beta}
		+T_{0,n}^\alpha P_n(e^{iA_n}-iA_n-I)P_n\left(T_{0}^{\beta}-T_{0,n}^{\beta}\right)\Bigg\} \Bigg]
	\end{align*}
	and hence we have the following estimate
	\begin{align}\label{eqapp3}
		\nonumber&\left\| \left[\Big\{p(T)-p(T_0)-\at{\dds}{s=0}p(T_s)\Big\}- P_n\Big\{p(T_n)-p(T_{0,n})-\at{\dds}{s=0}p(T_{s,n})\Big\}P_n\right]\right\|_1\\
		\nonumber&\leq ~~\sum_{\substack{\alpha+\beta=r\\\alpha\geq 0 ~\&~ \beta\geq 1}}\Bigg[\left\|\big( T^{\alpha}-T_{n}^{\alpha}- T_{0}^{\alpha}-T_{0,n}^{\alpha} \big)P_n\right\|_2 \left\|(e^{iA}-I)\right\|_2
		+ \left\|\big(T^{\alpha}-T_{0}^{\alpha}\big)\right\|_2 \left\|P_n^\perp (e^{iA}-I)\right\|_2\\ \nonumber&+\left\|\left(T_{n}^{\alpha}-T_{0,n}^{\alpha}\right)P_n\right\|_2 \left\|P_n\left( e^{iA}-e^{iA_n}\right)\right\|_2 + 2\left\|P_n\left( e^{iA_n}-I\right)P_n\right\|_2 \left\|P_n \left(T_0^{\beta}-T_{0,n}^{\beta}\right)\right\|_2 \\
		\nonumber& +\left\|(T_0^\alpha-T_{0,n}^{\alpha})P_n\right\|_2\left\|\left( e^{iA}-iA-I \right)\right\|_2 + \left\| P_n^\perp \left( e^{iA}-iA-I \right)\right\|_1\\
		& +\left\| P_n\left(e^{iA}-iA-e^{iA_n}+iA_n\right)\right\|_1
		+\left\| P_n(e^{iA_n}-iA_n-I)P_n\right\|_2 \left\|P_n\left(T_{0}^{\beta}-T_{0,n}^{\beta}\right)\right\|_2\Bigg].
	\end{align}
	Finally, using the estimates listed in  Lemma~\ref{appthm}, we conclude that each term in the right hand side of \eqref{eqapp3} converges to zero as $n\longrightarrow \infty$. By repeating the similar calculations as above and using  Lemma~\ref{appthm} and Remark \ref{adj} we conclude \eqref{eqapp} for $p(z)=z^r,r\leq -1$. This completes the proof. 
\end{proof}

\begin{rem}
	Note that, in the above Theorem~\ref{th3(2)} we prove the convergence of the expression in \eqref{eqapp} in trace norm instead of taking the trace of the expression and show the convergence. In other words, the above Theorem~\ref{th3(2)} deals with the trace norm convergence which is stronger in comparison with the trace convergence as obtained in \cite[Theorem 3.7]{CDP} and also we are dealing with pair of  contractions $(T,T_0)$ instead of pair unitaries $(U,U_0)$.  
\end{rem}

\section{\bf Existence of shift function in linear path}
Now we are in a position to derive the trace formula corresponding to the pair of contractions $(N,N_0)$, where $N_0$ is a normal operator. The following theorem is one of the main results in this section.
\begin{thm}\label{th4}
	Let $N$ and $N_0$ be two contraction operators in an infinite dimensional separable Hilbert space $\hil$ such that $N_0$ is normal and $V=N-N_0\in\hils$. Denote $N_s=N_0+sV, ~s\in[0,1]$. Then for any complex polynomial  $p(\cdot)$,
	$\Big\{p(N)-p(N_0)-\at{\dds}{s=0}\big\{p(N_s)\big\}\Big\}\in\boh $
	and there exists an $L^1(\mathbb{T})$-function $\eta$  (unique up to an analytic term) such that
	\begin{align}\label{eq20}
		\operatorname{Tr}\Bigg\{p(N)-p(N_0)-\at{\dds}{s=0} \big\{p(N_s)\big\} \Bigg\}=\int_{\cir}
		p''(z)\eta(z)dz.
	\end{align}
	Moreover, for every given $\epsilon >0$, we choose the function $\eta$ satisfying \eqref{eq20} in such a way so that 
	\begin{equation}\label{eq40}
		\|\eta\|_{L^1(\mathbb{T})}\leq \left(\frac{1}{2}+\epsilon\right) ~\|V\|_2^2. 
	\end{equation}
\end{thm}
\begin{proof}
	By Theorems \ref{th3(1)} and \ref{th2(1)}, we have that 
	\begin{align}\label{eq34}
		\nonumber &\operatorname{Tr}\Bigg\{p(N)-p(N_0)-\at{\dds}{s=0}\big\{p(N_s)\big\} \Bigg\}\\
		\nonumber&=\lim_{n\to \infty}\operatorname{Tr}~\Bigg\{P_n\Bigg(p(N_n)-p(N_{0,n})-\at{\dds}{s=0}\big\{p(N_{s,n})\big\} \Bigg)P_n\Bigg\}\\
		&=\lim_{n\to \infty}\int_{\cir} p''(z)\eta_n(z)dz,
	\end{align}
	$\text{where~} N_n=P_nNP_n,N_{0,n}=P_nN_0P_n,N_{s,n}=P_nN_sP_n$, and $\eta_n(z)$ is given by \eqref{aeq1}, that is 
	\begin{align}\label{eq30}
		\eta_n(z)=\int_{0}^{1}\operatorname{Tr} \Big[V_n\Big\{\mathcal{E}_{0,n}\Big(Arg(z)\Big)-\mathcal{E}_{s,n}\Big(Arg(z)\Big)\Big\}\Big]ds, ~z\in\cir,
	\end{align}
	where $\mathcal{E}_{s,n}(\cdot)$ and $\mathcal{E}_{0,n}(\cdot)$ are the semi-spectral measures corresponding to the contractions $N_{s,n}$ and $N_{0,n}$ respectively (see \eqref{spectralmeasurecon} and \eqref{funcalcon} in Section 2) and $V_n=P_nVP_n$. Moreover, from \eqref{eq35class} it follows that 
	\begin{equation}\label{eq36}
		\big\|[\eta_n]\big\|_{L^1(\mathbb{T})/H^1(\mathbb{T})}
		\leq \frac{1}{2} ~\|V_n\|_2^2.
	\end{equation}
	Next we show that the sequence $\{\eta_n\}$ converges in some suitable sense. Indeed, by following the idea contained in \cite{GePu, Ko, NH} (see also \cite{ChSi,CDP}), using the above expression \eqref{eq30} of $\eta_n$, using Fubini's theorem to interchange the orders of integration and integrating by-parts, we have for $f\in \mathcal{P}(\mathbb{T})$ and $g(e^{it})$ $=\int\limits_{0}^{t}f(e^{is})~ie^{is}ds,~t\in[0,2\pi]$ that
	\begin{align}\label{eq31}
		\nonumber &\int_{\cir} f(z)\big\{\eta_n(z)-\eta_m(z)\big\} dz
		=\int_{0}^{2\pi} f(e^{it})\big\{\eta_n(e^{it})-\eta_m(e^{it})\big\}~ie^{it}dt\\
		\nonumber &=\int_{0}^{2\pi}\ddt \big\{g(e^{it})\big\}\Bigg[\int_{0}^{1}\operatorname{Tr} \Big[V_n\Big\{\mathcal{E}_{0,n}(t)-\mathcal{E}_{s,n}(t)\Big\}-V_m\Big\{\mathcal{E}_{0,m}(t)-\mathcal{E}_{s,m}(t)\Big\}\Big]~ds\Bigg]~dt\\
		\nonumber &=\int_{0}^{1}\Bigg[\int_{0}^{2\pi}\ddt \big\{g(e^{it})\big\}~\operatorname{Tr} \Big[V_n\Big\{\mathcal{E}_{0,n}(t)-\mathcal{E}_{s,n}(t)\Big\}-V_m\Big\{\mathcal{E}_{0,m}(t)-\mathcal{E}_{s,m}(t)\Big\}\Big]dt\Bigg]ds\\
		\nonumber &=\int_{0}^{1}\Bigg[g(e^{it})~\operatorname{Tr} \Big[V_n\Big\{\mathcal{E}_{0,n}(t)-\mathcal{E}_{s,n}(t)\Big\}-V_m\Big\{\mathcal{E}_{0,m}(t)-\mathcal{E}_{s,m}(t)\Big\}\Big]\Bigg\lvert_{t=0}^{2\pi}\\
		\nonumber &\hspace*{1in}-\int_{0}^{2\pi}g(e^{it})~\operatorname{Tr} \Big[V_n\Big\{\mathcal{E}_{0,n}(dt)-\mathcal{E}_{s,n}(dt)\Big\}-V_m\Big\{\mathcal{E}_{0,m}(dt)-\mathcal{E}_{s,m}(dt)\Big\}\Big]\Bigg]ds\\
		\nonumber &=-\int_{0}^{1}ds\int_{0}^{2\pi}g(e^{it})~\operatorname{Tr} \Big[V_n\Big\{\mathcal{E}_{0,n}(dt)-\mathcal{E}_{s,n}(dt)\Big\}-V_m\Big\{\mathcal{E}_{0,m}(dt)-\mathcal{E}_{s,m}(dt)\Big\}\Big]\\
		\nonumber &=\int_{0}^{1}ds~\operatorname{Tr} \Big[V_n\Big\{g(N_{s,n})-g(N_{0,n})\Big\}-V_m\Big\{g(N_{s,m})-g(N_{0,m})\Big\}\Big]\\
		\nonumber &=\int_{0}^{1}ds~\operatorname{Tr} \Bigg[V_n\Bigg\{\Big\{g(N_{s,n})-g(N_{s})\Big\}-\Big\{g(N_{0,n})-g(N_{0})\Big\}\Bigg\}\\
		&\hspace*{.1in}-V_m\Bigg\{\Big\{g(N_{s,m})-g(N_{s})\Big\}-\Big\{g(N_{0,m})-g(N_{0})\Big\}\Bigg\}+(V_n-V_m)\Big\{g(N_{s})-g(N_{0})\Big\}\Bigg].
	\end{align}
	On the other hand using Theorem~\ref{estimatethm} and using  trace properties we obtain for $s\in[0,1]$ that 
	\begin{align}
		\nonumber & \Bigg|\operatorname{Tr} \Big[V_n\big\{g(N_{s,n})-g(N_{s})\big\}\Big]\Bigg| =\Bigg|\operatorname{Tr}
		\Big[V_nP_n\big\{g(N_{s})-g(N_{s,n})\big\}P_n\Big]\Bigg|\\
		\nonumber&\leq~ \norm{V_n}_2~\left\|P_n\Big\{g(N_{s})-g(N_{s,n})\Big\}P_n\right\|_2
		\leq~ \norm{V}_2~\left\|P_n\Big\{g(N_{s})P_n-P_n~g(N_{s,n})\Big\}P_n\right\|_2\\
		\nonumber &\leq~ \norm{V}_2~\Big\|g(N_{s}) P_n - P_n~g(N_{s,n})\Big\|_2
		\leq~\norm{f}_\infty\norm{V}_2~\big\|N_{s} P_n - P_nN_{s,n}\big\|_2\\
		\label{eq6} & = \norm{f}_\infty\norm{V}_2~\big\|P_n^{\perp}N_{s} P_n\big\|_2 ~\leq \norm{f}_\infty\norm{V}_2~\left(\big\|P_n^{\perp}N_{0} P_n\big\|_2 + \big\|P_n^{\perp}V\big\|_2\right).
	\end{align}
	Similarly, by repeating the above calculations we also obtain
	\begin{align}
		\nonumber &\Bigg|\operatorname{Tr} \Big[(V_n-V_m)\big\{g(N_{s})-g(N_{0})\big\}\Big]\Bigg|\leq~\norm{f}_\infty\norm{V}_2~\big\|V_n-V_m\big\|_2\\
		\label{eq7} & \leq ~\norm{f}_\infty\norm{V}_2~\left(\big\|P_n^{\perp}V\big\|_2 +\big\|VP_n^{\perp}\big\|_2 +\big\|P_m^{\perp}V\big\|_2 + \big\|VP_m^{\perp}\big\|_2\right).
	\end{align}
	Now combining \eqref{eq31}, \eqref{eq6} and \eqref{eq7} we get
	\begin{align}\label{eq32}
		\nonumber & \left|\int_{\cir} f(z)\big\{\eta_n(z)-\eta_m(z)\big\}dz\right| \\
		\nonumber & \leq \int_{0}^{1}ds~\Bigg|\operatorname{Tr} \Bigg[V_n\Bigg\{\Big\{g(N_{s,n})-g(N_{s})\Big\}-\Big\{g(N_{0,n})-g(N_{0})\Big\}\Bigg\}\\
		\nonumber &\hspace*{.1in}-V_m\Bigg\{\Big\{g(N_{s,m})-g(N_{s})\Big\}-\Big\{g(N_{0,m})-g(N_{0})\Big\}\Bigg\}+(V_n-V_m)\Big\{g(N_{s})-g(N_{0})\Big\}\Bigg]\Bigg|\\
		&\leq K_{m,n}~\norm{f}_\infty\norm{V}_2,
	\end{align}
	where
	\begin{align*}
		& K_{m,n}= ~\Big\{ 2\left(\big\|P_n^{\perp}N_{0} P_n\big\|_2 + \big\|P_n^{\perp}V\big\|_2\right) + 
		2\left(\big\|P_m^{\perp}N_{0} P_m\big\|_2 + \big\|P_m^{\perp}V\big\|_2\right) \\
		& \hspace{1in} + \left(\big\|P_n^{\perp}V\big\|_2 +\big\|VP_n^{\perp}\big\|_2 +\big\|P_m^{\perp}V\big\|_2 + \big\|VP_m^{\perp}\big\|_2\right)\Big\}.
	\end{align*}
	Therefore using Lemma~\ref{supnorm} and the above estimate \eqref{eq32}  we conclude
	\begin{align*}
		\big\|[\eta_n]-[\eta_m]\big\|_{L^1(\mathbb{T})/H^1(\mathbb{T})}=&\big\|[\eta_n-\eta_m]\big\|_{L^1(\mathbb{T})/H^1(\mathbb{T})}\\
		=&~\sup_{\substack{f\in \mathcal{P}(\cir);~\|f\|_{\infty}\leq 1}}\Bigg|\int_{\cir} f(z)\big\{\eta_n(z)-\eta_m(z)\big\}dz\Bigg|\\
		\leq~&~K_{m,n}~\|V\|_2\longrightarrow 0\quad \text{as}\quad m,n\longrightarrow \infty,
	\end{align*}
	by using Theorem \ref{appthm}
	and hence $\Big\{[\eta_n]\Big\}$ is a Cauchy sequence in $L^1(\mathbb{T})/H^1(\mathbb{T})$. Consequently, there exists a $\eta \in L^1(\mathbb{T})$ such that $\Big\{[\eta_n]\Big\}$ converges to $[\eta]$ in $L^1(\mathbb{T})/H^1(\mathbb{T})$-norm, that is
	\begin{equation*}
		\lim_{n\longrightarrow \infty}\big\|[\eta_n]-[\eta]\big\|_{L^1(\mathbb{T})/H^1(\mathbb{T})}=0,
	\end{equation*} 
	which in particular implies that
	\begin{align}\label{eq33}
		\lim_{n\longrightarrow \infty} \int_{\cir} p(z)\eta_n(z)dz = \int_{\cir} p(z)\eta(z)dz
	\end{align}
	for all complex polynomials $p(\cdot)$.
	Therefore combining \eqref{eq34} and \eqref{eq33} we get
	\begin{align*}
		\operatorname{Tr}\Big\{p(N)-p(N_0)-\at{\dds}{s=0}p(N_s) \Big\}=\lim_{n\to\infty}\int_{\cir} p''(z)\eta_n(z)dz
		=\int_{\cir} p''(z)\eta(z)dz.
	\end{align*}
	Furthermore, the equation \eqref{eq36} yields
	\begin{equation*}
		\big\|[\eta]\big\|_{L^1(\mathbb{T})/H^1(\mathbb{T})}
		\leq \frac{1}{2} ~\|V\|_2^2,
	\end{equation*}
	which by applying the definition of the $L^1(\mathbb{T})/H^1(\mathbb{T})$- norm, for every $\epsilon>0$, there is a function $\eta \in L^1(\mathbb{T})$
	such that
	\begin{equation*}
		\big\|\eta\big\|_{L^1(\mathbb{T})}
		\leq \left(\frac{1}{2}+\epsilon\right) ~\|V\|_2^2.
	\end{equation*} 
	This completes the proof. 
\end{proof}


	

Our next aim is to extend the class of functions for which the trace formula  \eqref{eq20} holds. For that we need the following lemma. The proof the following lemma is similar to the proof of \cite[Lemma 4.2]{CDP}, so we state it without proof.

\begin{lem}\label{l4}
	Let $T$ and $T_0$ be two contractions in a separable infinite dimensional Hilbert space $\hil$ and $f\in\mathcal{F}_2^+(\cir)$. Let $T_s=T_0+s(T-T_0), ~s\in[0,1]$, then 
	\begin{equation}\label{eq37}
		\at{\dds}{s=0}f(T_s)=\sum_{k=0}^{\infty}\hat{f}(k)\at{\dds}{s=0} T_s^k.	\end{equation}
\end{lem}
Now we are in a position to prove our main result in this section.
\begin{thm}\label{th5}
	Let $N$ be a contraction and $N_0$ be a normal contraction in an infinite dimensional separable Hilbert space $\hil$ such that $V=N-N_0\in\hils$. Denote $N_s=N_0 + sV$, ~$s\in [0,1]$. Then for any $\Phi\in\mathcal{F}_2^+(\cir)$, $\Big\{ \Phi(N)-\Phi(N_0)-\at{\dds}{s=0}\Phi(N_s)\Big\}\in\boh$
	and there exists an $L^1(\mathbb{T})$-function $\eta$  (unique up to an analytic term) such that
	\begin{align}\label{eq38}
		\operatorname{Tr}\Bigg\{\Phi(N)-\Phi(N_0)-\at{\dds}{s=0} \big\{\Phi(N_s)\big\} \Bigg\}=\int_{\cir}
		\Phi''(z)\eta(z)dz.
	\end{align}
	Furthermore, for every given $\epsilon >0$, we choose the function $\eta$ satisfying \eqref{eq38} in such a way so that 
	\[
	\|\eta\|_{L^1(\mathbb{T})}\leq \left(\frac{1}{2}+\epsilon\right) ~\|V\|_2^2. 
	\]
	
\end{thm}
\begin{proof} Using the above Lemma \ref{l4}, the proof follows from the similar argument as given in \cite[Theorem 4.3]{CDP}.
\end{proof}
\begin{cor}\label{cr1}
If $U$ and $U_0$ are two unitary operators in an infinite dimensional separable Hilbert space $\hil$ such that $U-U_0\in\hils$. Denote $U_s=U_0 + sV$, ~$s\in [0,1],$ and $U=U_1$. Then there exists a $L^1(\cir)$-function $\eta$ (unique up to an analytic term) and satisfying the equation \eqref{eq40} such that, for any $z\in\mathbb{C}$ with $|z|> 1$,
\begin{align*}
	\operatorname{Tr}\Bigg\{(U-z)^{-1}-(U_0-z)^{-1}-\at{\dds}{s=0}(U_s-z)^{-1} \Bigg\}=~\int_\cir~ \frac{d^2}{dw^2}\Big\{(w-z)^{-1}\Big\} \eta(w)dw.
\end{align*}
\end{cor}

\section{\bf Trace formula for contractions}
In our previous section, by finite-dimensional approximation method, we derive trace formula for certain pair of contractions in linear path with the initial contraction is normal. However, to relax the normality condition on the initial contraction it is natural to use Sch\"{a}ffer matrix dilation. Therefore, in this section we prove the trace formula for class of pair of contractions $(T,T_0)$ such that $T-T_0\in \mathcal{B}_2(\mathcal{H})$ using the Sch\"{a}ffer matrix unitary dilation. Note that Sch\"{a}affer matrix unitary dilation was substantially used in \cite[Section 9]{MNP19}, especially in the proofs of \cite[Theorems 9.4, 9.6, and 9.7]{MNP19} to study Krein trace formula. Let $T$ be a contraction on an infinite dimensional separable Hilbert space $\hil$. Now we need the construction of the Sch\"{a}ffer matrix unitary dilation $U_T$ of $T$ on the two-sided sequence space $l^2_{\mathbb{Z}}(\mathcal{H})$ of $\mathcal{H}$-valued sequences (that is,
on the Hilbert space $l^2_{\mathbb{Z}}(\mathcal{H}):=\bigoplus\limits_{n=-\infty}^{-1}\mathcal{H}\bigoplus\hil\bigoplus\limits_{n=1}^{\infty}\mathcal{H}$) (see \cite{NGFO}, chapter I, section 5). Note that we embed $\mathcal{H}$ in $l^2_{\mathbb{Z}}(\mathcal{H})$ by identifying the element $h\in \mathcal{H}$ with the vector $\{h_n\}_{n\in \mathbb{Z}}\in l^2_{\mathbb{Z}}(\mathcal{H})$ for which $h_0=h$ and $h_n=0$ ($n\neq 0$). Then $\mathcal{H}$ becomes a subspace of $l^2_{\mathbb{Z}}(\mathcal{H})$, and the orthogonal projection from  $l^2_{\mathbb{Z}}(\mathcal{H})$ into $\mathcal{H}$ is given by 
\begin{equation}\label{orthoprojection}
P_{\mathcal{H}}(\{h_n\}_{n\in \mathbb{Z}})=h_0.
\end{equation}
It is important to note that such a dilation does not have to be minimal but at the same time the advantage of this dilation is that it allows us to consider unitary dilations of contractions on $\mathcal{H}$ on the same Hilbert space $l^2_{\mathbb{Z}}(\mathcal{H})$. The following is the block matrix representation of $U_T$
\begin{equation}\label{shaffermatrix}
U_{_T} = \begin{bmatrix} 
	\ddots&\vdots&\vdots&\vdots&\vdots&\vdots&\vdots&\vdots& \reflectbox{$\ddots$} \\
	\cdots&0 &0 &I &0 &0 &0 &0 &\cdots  \\
	\cdots&0 &0 &0 &D_T &-T^* &0 &0 &\cdots \\
	\cdots&0 &0 &0 &\boxed{T} &D_{T^*} &0 &0 &\cdots \\
	\cdots&0 &0 &0 &0 &0 & I&0&\cdots\\
	\cdots&0 &0 &0 &0 &0 & 0&I&\cdots\\
	\reflectbox{$\ddots$}&\vdots &\vdots&\vdots&\vdots&\vdots&\vdots&\vdots& \ddots& 
\end{bmatrix},
\end{equation}
where $D_T=(I-T^*T)^{1/2}$ and $D_{T^*}=(I-TT^*)^{1/2}$ are the defect operators corresponding to the contractions $T$ and $T^*$ respectively. In the block matrix representation \eqref{shaffermatrix} of $U_T$, the entry $T$ is at the $(0,0)$ position and  the $(i,j)$-th entries $U_{_T}(i,j)$ of  $U_{_T}$ are given by 
\[
U_{_T}(0,0)=T,~U_{_T}(-1,0)=D_{T},~U_{_T}(-1,1)=-T^*,~U_{_T}(0,1)=D_{T^*},~U_{_T}(j,j+1)=I
\] 
for $j\neq 0,-1$, while all the remaining entries are equal to zero. The following lemma is essential to prove our main results in this section.

\begin{lem}\label{th6}
Let $T$ and $T_0$ be two contractions in an infinite dimensional separable Hilbert space $\mathcal{H}$ such that $V=T-T_0\in\hils$. Let $T_s=T_0+sV,~s\in[0,1]$. Then for any complex polynomial $p(\cdot),\hspace*{.1in}\Bigg\{p(T)-p(T_0)-\at{\dds}{s=0} \big\{p(T_s)\big\}\Bigg\}\in\boh$ and
\begin{align}\label{eqseclast}
	\operatorname{Tr}\Bigg\{p(T)-p(T_0)-\at{\dds}{s=0} \big\{p(T_s)\big\} \Bigg\}=\lim_{t\to 0}\operatorname{Tr}\Bigg\{p(T)-p(T_0)-\frac{p(T_t)-p(T_0)}{t} \Bigg\}.
\end{align}
\end{lem}
\begin{proof}
It will be sufficient to prove the theorem for $p(z)=z^r$. Note that for $r=0$ or $1$, both sides of \eqref{eqseclast} are identically zero.  Now by using similar kind of expressions as in \eqref{exp1} and \eqref{exp2}, we conclude  $\Bigg\{p(T)-p(T_0)-\at{\dds}{s=0} p(T_s)\Bigg\}, \left\{p(T)-p(T_0)-\frac{p(T_t)-p(T_0)}{t} \right\}\in\boh$ for $t\in [0,1]$. Furthermore, 
\begin{align*}
	&\Bnorm{~\Big\{p(T)-p(T_0)-\frac{p(T_t)-p(T_0)}{t} \Big\}-\Big\{p(T)-p(T_0)-\at{\dds}{s=0} p(T_s)\Big\}~}_1\\
	=~&\Bnorm{~\Big\{T^r-T_0^r-\frac{T_t^r-T_0^r}{t} \Big\}-\Big\{T^r-T_0^r-\sum_{j=0}^{r-1} T^{r-j-1}_{0}VT^j_0\Big\}~}_1\\
	=~&\Bnorm{~\sum_{j=0}^{r-1} T^{r-j-1}_{t}VT^j_0-\sum_{j=0}^{r-1} T^{r-j-1}_{0}VT^j_0~}_1
	=~\Bnorm{~\sum_{j=0}^{r-2}\sum_{k=0}^{r-j-2} T^{r-j-k-2}_{t}tVT_0^kVT^j_0~}_1\\
	\leq~&~|t|~\sum_{j=0}^{r-2}\sum_{k=0}^{r-j-2}\norm{T^{r-j-k-2}_{t}VT_0^k}_2\norm{VT^j_0}_2
	\leq~~|t|~\sum_{j=0}^{r-2}\sum_{k=0}^{r-j-2}\norm{V}_2^2\to 0~\text{as}~t\to 0,
\end{align*}
and hence \eqref{eqseclast} follows. This completes the proof.
\end{proof}

Motivated from the work of Marcantognini and Mor$\acute{a}$n	\cite{MaMo} we have the following one of the main result in this section.
\begin{thm}\label{th8}
Let $T$ and $T_0$ be two contractions in an infinite dimensional separable Hilbert space $\mathcal{H}$ such that
\vspace{0.1in}

$(i)$~~ $T-T_0\in\hils$, $(ii)$~~ $\dim \ker(T_0)=\dim \ker(T_0^*)$, and
$(iii)$~~ $D_{T_0}\in\hils$.
Denote $T_s=T_0+s(T-T_0), ~s\in[0,1]$. Then for any complex polynomial  $p(\cdot)$,
$\Bigg\{p(T)-p(T_0)-\at{\dds}{s=0}\big\{p(T_s)\big\}\Bigg\}\in\boh $ and there exists an $L^1(\mathbb{T})$-function $\eta$  (unique up to an analytic term) such that
\begin{align}\label{11}
	\operatorname{Tr}\Bigg\{p(T)-p(T_0)-\at{\dds}{s=0} \big\{p(T_s)\big\} \Bigg\}=\int_{\cir}
	p''(z)\eta(z)dz.
\end{align}
Moreover, the function $\eta$ satisfying \eqref{11} and satisfies the equation \eqref{eq20}.
\end{thm}
\begin{proof}
Let $U_{T_0}$ be the minimal Sch\"{a}ffer matrix unitary dilation of $T_0$ on the minimal dilation space $\mathcal{K}:=l^2_{\mathbb{N}}(\mathcal{D}_{T_0})\oplus\hil\oplus l^2_{\mathbb{N}}(\mathcal{D}_{T_0^*})$ and we have the same block matrix representation of $U_{T_0}$ as in \eqref{shaffermatrix} on $\mathcal{K}$. Let $T_0=V_{T_0}|T_0|$ be the polar decomposition of $T_0$, so that $|T_0|=(T_0^*T_0)^{1/2}$ and $V_{T_0}$ is an isometry from $\overline{\operatorname{Ran}(T_0^*)}$ onto $\overline{\operatorname{Ran}(T_0)}$. 
Since $\dim \ker(T_0)=\dim \ker(T_0^*)$, then we can extend $V_{T_0}$ to a unitary operator on the full space $\hil$. Now onward we assume $V_{T_0}$ is a unitary operator on $\hil$ such that $T_0=V_{T_0}|T_0|$. Next we note that
\begin{align}\label{13}
	V_{T_0}D_{T_0}=D_{T_0^*}V_{T_0} \text{ and } V_{T_0}-T_0=V_{T_0}(1-|T_0|)=V_{T_0}(1-T_0^*T_0)(1+|T_0|)^{-1}.
\end{align}
Now we extend $T$ to a contraction ${\tilde{U}_T}$ on $\mathcal{K}$ be setting

\begin{equation}\label{12}
	{\tilde{U}_T} = \begin{bmatrix} 
		\ddots&\vdots&\vdots&\vdots&\vdots&\vdots&\vdots&\vdots& \reflectbox{$\ddots$} \\
		\cdots&0 &0 &I &0 &0 &0 &0 &\cdots  \\
		\cdots&0 &0 &0 &0 &-V^*_{T_0} &0 &0 &\cdots \\
		\cdots&0 &0 &0 &\boxed{T} &0 &0 &0 &\cdots \\
		\cdots&0 &0 &0 &0 &0 & I&0&\cdots\\
		\cdots&0 &0 &0 &0 &0 & 0&I&\cdots\\
		\reflectbox{$\ddots$}&\vdots &\vdots&\vdots&\vdots&\vdots&\vdots&\vdots& \ddots& 
	\end{bmatrix}.
\end{equation}
Note that in the above block matrix representation \eqref{12} of $U_T$, the $(i,j)$-th entries ${\tilde{U}_T}(i,j)$ of  ${\tilde{U}_T}$ are given by 
\[
{\tilde{U}_T}(0,0)=T,~{\tilde{U}_T}(-1,0)=0,~{\tilde{U}_T}(-1,1)=-V^*_{T_0},~{\tilde{U}_T}(0,1)=0,~{\tilde{U}_T}(j,j+1)=I
\] 
for $j\neq 0,-1$, while all the remaining entries are equal to zero. Therefore we have
\begin{align*}
	T^n=\tilde{P}_{\mathcal{H}}{\tilde{U}_T}^n|_{\mathcal{H}}\quad \text{and}\quad T_0^n=\tilde{P}_{\mathcal{H}}U_{T_0}^n|_{\mathcal{H}}\quad \text{for}\quad n\geq 1,
\end{align*}
where $\tilde{P}_\hil$ is the orthogonal projection of $\mathcal{K}$ onto $\hil$.
By hypothesis $(i)$, $(iii)$ and using the relation \eqref{13} we conclude ${\tilde{U}_T}-U_{T_0}\in\mathcal{B}_2(\mathcal{K})$.
Denote $U_{t,T}=(1-t)U_{T_0}+t{\tilde{U}_T}=U_{T_0}+t({\tilde{U}_T}-U_{T_0})$ for $t\in[0,1]$. Note that $p({\tilde{U}_T})$ and $p(U_{T_0})$ are upper triangular matrices with the only nonzero diagonal entries
$p(T)$ and $p(T_0)$. Thus $\Bigg\{p({\tilde{U}_T})-p(U_{T_0})-\frac{p(U_{t,T})-p(U_{T_0})}{t} \Bigg\}$ is an upper triangular matrix with the only nonzero diagonal entry $\left\{p(T)-p(T_0)-\frac{p(T_t)-p(T_0)}{t} \right\}$ and hence
\begin{equation}\label{17}
	\operatorname{Tr}\Bigg\{p(T)-p(T_0)-\frac{p(T_t)-p(T_0)}{t} \Bigg\}
	= \operatorname{Tr}\Bigg\{p({\tilde{U}_T})-p(U_{T_0})-\frac{p(U_{t,T})-p(U_{T_0})}{t} \Bigg\}\quad \text{for} \quad t\in [0,1].
\end{equation}
Therefore by applying Lemma~\ref{th6} corresponding to  pair of contractions $(T,T_0)$ and $({\tilde{U}_T},U_{T_0})$ and using \eqref{17} we get
\begin{align*}
	&\operatorname{Tr}\Bigg\{p(T)-p(T_0)-\at{\dds}{s=0} p(T_s) \Bigg\}
	=~\lim_{t\to 0}\operatorname{Tr}\Bigg\{p(T)-p(T_0)-\frac{p(T_t)-p(T_0)}{t} \Bigg\}\\
	=&~\lim_{t\to 0}\operatorname{Tr}\Bigg\{p({\tilde{U}_T})-p(U_{T_0})-\frac{p(U_{t,T})-p(U_{T_0})}{t} \Bigg\} = \operatorname{Tr}\Bigg\{p({\tilde{U}_T})-p(U_{T_0})-\at{\dds}{s=0} p(U_{s,T}) \Bigg\},
\end{align*} 
which by applying Theorem~\ref{th4} corresponding to the pair  $({\tilde{U}_T},U_{T_0})$ we conclude that there exists an $L^1(\mathbb{T})$-function $\eta$  (unique up to an analytic term) satisfying \eqref{eq20} such that
\begin{align*}
	\operatorname{Tr}\Bigg\{p(T)-p(T_0)-\at{\dds}{s=0} \big\{p(T_s)\big\} \Bigg\}=\int_{\cir}
	p''(z)\eta(z)dz.
\end{align*}
This completes the proof.
\end{proof}

The following theorem is the second main result in this section in which we prove the trace formula for a class of pairs of contractions different from the class mentioned in  Theorem~\ref{th8}.

\begin{thm}\label{th7}
Let $T$ and $T_0$ be two contractions in an infinite dimensional separable Hilbert space $\mathcal{H}$ such that $U_{_{T}}-U_{_{T_0}}\in\mathcal{B}_2(l^2_{\mathbb{Z}}(\mathcal{H}))$. Denote $T_s=T_0+s(T-T_0), ~s\in[0,1]$. Then for any complex polynomial  $p(\cdot)$,
$\Bigg\{p(T)-p(T_0)-\at{\dds}{s=0}\big\{p(T_s)\big\}\Bigg\}\in\boh$
and there exists an $L^1(\mathbb{T})$-function $\eta$  (unique up to an analytic term) such that
\begin{align}\label{eqlast20}
	\operatorname{Tr}\Bigg\{p(T)-p(T_0)-\at{\dds}{s=0} \big\{p(T_s)\big\} \Bigg\}=\int_{\cir}
	p''(z)\eta(z)dz.
\end{align}
Moreover, the function $\eta$ satisfying \eqref{eqlast20}, and also satisfies the  estimate \eqref{eq40} with $V:=T-T_0$.
\end{thm}

\begin{proof}
Note that $U_T$ and $U_{T_0}$ are the Sch\"{a}ffer matrix unitary dilation of $T$ and $T_0$ respectively on the same Hilbert space $l^2_{\mathbb{Z}}(\mathcal{H})$ (see \eqref{shaffermatrix}), that is
\begin{equation}\label{dilationeq}
	T^n=P_{\mathcal{H}}U_T^n|_{\mathcal{H}}\quad \text{and}\quad T_0^n=P_{\mathcal{H}}U_{T_0}^n|_{\mathcal{H}}\quad \text{for}\quad n\geq 1,
\end{equation}
where $P_{\mathcal{H}}$ as in \eqref{orthoprojection}.
Since $U_{_{T}}-U_{_{T_0}}\in\mathcal{B}_2(l^2_{\mathbb{Z}}(\mathcal{H}))$, then from \eqref{dilationeq} we conclude $T-T_0\in \mathcal{B}_2(\mathcal{H})$. Denote $U_{t,T}=(1-t)U_{T_0}+tU_T=U_{T_0}+t(U_T-U_{T_0})$ for $t\in[0,1]$. Then by the similar argument as in Theorem~\ref{th8} we conclude the theorem.
\end{proof}
\begin{rem}
	In the above proof of Theorem~\ref{th4} and hence of Theorem \ref{th8}, and Theorem~\ref{th7}, due to the finite-dimensional approximation method, we have the existence of a sequence $\{\eta_n\}$ of integrable functions with explicit expressions in terms of the semi-spectral measures converging to our desired spectral shift function in an appropriate sense. Note that in the proof of Theorem \ref{PSthm} (\cite[Theorem 1]{PoSu}), the sequence $\{\eta_n\}$ has no explicit expressions in terms of the semi-spectral measures. In particular, in finite-dimensional Hilbert space we have explicit expressions of shift functions as done in the case of Krein trace formula for pair of self-adjoint operators by Voiculescu \cite{Voi}, Sinha and Mohapatra \cite{MoSi94}.  Moreover, our approach gives a better bound \eqref{eq40} for the shift function compared to \eqref{eqin40}. 
\end{rem}
\begin{rem}
In a similar spirit as in Theorem \ref{th5}, we also prove Theorem~\ref{th7} corresponding to the class $\mathcal{F}_2^+(\cir)$.
\end{rem}
\begin{rem}
If one of $T$ and $T_0$ is a strict contraction such that $T-T_0\in \hils$, then the conclusion of Theorem~\ref{th7} is also true. Indeed, we can show as in \cite{ChSiCont} that the difference of the corresponding Sch\"{a}ffer matrix unitary dilation is also a Hilbert-Schmidt operator, that is $U_T-U_{T_0}\in \hils$. 
\end{rem}

\section{\bf Trace formula for self-adjoint operators}
In this section, motivated from the work of Neidhardt in \cite[Section 3]{NH} we prove the trace formula, using one of our main results in earlier section, for a pair of self-adjoint operators $H_0$ and $H_1$ on an infinite-dimensional Hilbert space $\mathcal{H}$ such that the difference $(H_1-i)^{-1}-(H_0-i)^{-1}\in \mathcal{B}_2(\mathcal{H})$ via Cayley transform. 
Let us consider the following class of functions
\[\mathcal{F}_\mathbb{R}:=\left\{ \psi:\mathbb{R}\to\mathbb{C} \text{ such that } \psi(\lambda)=\phi\left(\dfrac{i-\lambda}{i+\lambda}\right) \text{ for some } \phi\in\mathcal{F}_2^+(\cir)  \right\}.\] 
Let $H$ and $H_0$ be two arbitrary self-adjoint operators in $\hil$ such that $Dom(H) = Dom(H_0)$, and let
\begin{equation}\label{CTrans}
U =(i-H)(i+H)^{-1} \quad \text{and} \quad U_0=(i-H_0)(i+H_0)^{-1}
\end{equation}  
be the corresponding unitary operators obtained via the  Cayley transform of $H$ and $H_0$ respectively. Let $\psi(\lambda)=\phi\left(\dfrac{i-\lambda}{i+\lambda}\right)\in\mathcal{F}_\mathbb{R}$ for some $\phi\in\mathcal{F}_2^+(\cir)$, $U_s=(1-s)U_0+sU$, and $H_s=sH_0+(1-s)H$ for $s\in[0,1]$. Then it is easy to observe that
\begin{align*}
\psi(H)=\phi(U),~\psi(H_0)=\phi(U_0), \text{ and } \psi(W_s)=\phi(U_s),
\end{align*}
where 
\begin{align*}
W_s=\Big((H+i)(H_s+i)^{-1}(H_0+i)-i\Big),
\end{align*} 
and hence we have the following operator equality
\begin{align}\label{eqself}
\psi(H)-\psi(H_0)-\dfrac{d}{ds}\Big|_{s=0}\psi(W_s)
=~\phi(U)-\phi(U_0)-\dfrac{d}{ds}\Big|_{s=0}\phi(U_s).
\end{align}

In the following, we denote $R_z=(H-z)^{-1}$ for $z\in \mathbb{C}$ as the resolvent operator corresponding to an unbounded self-adjoint operator $H$ and $\rho(H)$ is the associated resolvent set. Moreover, we also denote $Dom(H)$ as the domain of definition of the self-adjoint operator $H$ (possibly unbounded). The following is the main result in this section.

\begin{thm}\label{selfth}
Let $H$ and $H_0$ be two self-adjoint operators in a separable infinite dimensional Hilbert space $\hil$ such that $H-H_0\in\hils$. Let $$W_s=\Big((H+i)(H_s+i)^{-1}(H_0+i)-i\Big),$$ where $H_s=sH_0+(1-s)H,~s\in [0,1] $. Then there exists a measurable function $\xi:\mathbb{R}\to\mathbb{C}$ obeying $(1+\lambda^2)^{-1}\xi\in L^1(\mathbb{R})$ such that
\begin{align}\label{selfuni}
	\operatorname{Tr}\left\{\psi(H)-\psi(H_0)-\dfrac{d}{ds}\Big|_{s=0}\psi(W_s)\right\}=\int_{-\infty}^{\infty}~\dfrac{d}{d\lambda}\Big\{(1+\lambda^2)\psi'(\lambda)\Big\}~\xi(\lambda)~d\lambda 
\end{align} 
for each $\psi\in\mathcal{F}_\mathbb{R}$.
In particular, for all $z\in\mathbb{C}$ with $\operatorname{Im}(z)< 0$,
\begin{align}\label{res}
	\operatorname{Tr}\Bigg\{ (H-z)^{-1}-(H_0-z)^{-1}-\frac{i+H_0}{H_0-z}~M~\frac{i+H_0}{H_0-z}\Bigg\}=~\int_{-\infty}^{\infty} ~\dfrac{2(1+\lambda z)}{(\lambda-z)^3} ~\xi(\lambda)~d\lambda,
\end{align}
where $M=R_{-i}-R^0_{-i}$.
\end{thm}
\begin{proof}
Let $U$ and $U_0$ be as in \eqref{CTrans}. Since $R_z-R^0_z\in\hils$, then it immediately follows that $U-U_0\in\hils$. Therefore using the above identity \eqref{eqself} and using Theorem \ref{th5} we conclude that there exists an $L^1(\mathbb{T})$-function $\eta$  (unique up to an analytic term) such that 
\begin{align}\label{1}
	\nonumber&\operatorname{Tr}\left\{\psi(H)-\psi(H_0)-\dfrac{d}{ds}\Big|_{s=0}\psi\left(W_s\right)\right\}\\
	\nonumber&=\operatorname{Tr}\left\{\phi(U)-\phi(U_0)-\dfrac{d}{ds}\Big|_{s=0}\phi(U_s)\right\} =\int_{\cir} \phi''(z)\eta(z)dz\\
	&=\int_{\cir} \phi''(z)\left(\eta(z)-\dfrac{z}{2\pi i}\int_\cir \dfrac{\eta(\omega)}{\omega^2}d\omega\right)dz
	=\int_{\cir} \phi''(z)\Gamma(z)dz,
\end{align}
where 
$$\Gamma(z)=\eta(z)-\dfrac{z}{2\pi i}\int_{\cir} \dfrac{\eta(\omega)}{\omega^2}d\omega, \quad z\in \mathbb{T}.$$
Next by using change of variables $z=e^{it}$, performing integration by-parts and using the fact that 
$
\int_{0}^{2\pi}e^{-is}\Gamma(e^{is})ds = 0
$
we get 
\begin{align}\label{4}
	\nonumber&\operatorname{Tr}\left\{\psi(H)-\psi(H_0)-\dfrac{d}{ds}\Big|_{s=0}\psi\left(W_s\right)\right\}\\
	&\nonumber = \int_{0}^{2\pi}\dfrac{d^2}{dt^2}\{\phi(e^{it})\}\left(-ie^{-it}\right)\Gamma(e^{it})dt-\int_{0}^{2\pi}\ddt\{\phi(e^{it})\}e^{-it}\Gamma(e^{it})dt\\
	\nonumber & =\int_{0}^{2\pi}\dfrac{d^2}{dt^2}\{\phi(e^{it})\}\left[\left(-ie^{-it}\right)\Gamma(e^{it})+\left(\int_{0}^{t}e^{-is}\Gamma(e^{is})ds\right) \right]dt\\
	& =\int_{0}^{2\pi}\dfrac{d^2}{dt^2}\{\phi(e^{it})\}~\tilde{\eta}(e^{it})~dt,
\end{align}
where $$\tilde{\eta}(e^{it})=\left(-ie^{-it}\right)\Gamma(e^{it})+\left(\int_{0}^{t}e^{-is}\Gamma(e^{is})ds\right), ~t\in[0,2\pi].$$
Clearly $\tilde{\eta}\in L^1([0,2\pi])$ and again by using change of variables  $e^{it}=\dfrac{i-\lambda}{i+\lambda}$, from \eqref{4} we conclude
\begin{align*}
	\operatorname{Tr}\left\{\psi(H)-\psi(H_0)-\dfrac{d}{ds}\Big|_{s=0}\psi\left(W_s\right)\right\}
	&=\int_{-\infty}^{\infty}~\dfrac{d}{d\lambda}\Big\{ (1+\lambda^2)\psi'(\lambda)\Big\}~\xi(\lambda)~d\lambda,
\end{align*}
where $\xi(\lambda)=\dfrac{1}{2}\tilde{\eta}\left(\dfrac{i-\lambda}{i+\lambda}\right), \lambda\in\mathbb{R}$, and moreover $$\int_{-\infty}^{\infty}\dfrac{|\xi(\lambda)|}{1+\lambda^2}~d\lambda=\dfrac{1}{4}\int_{0}^{2\pi}|\tilde{\eta}(e^{it})|~dt<\infty.$$
Next, we denote $\tau=~\frac{i-z}{i+z}$ for $z\in\mathbb{C}$ with $\operatorname{Im}(z)< 0$. Then it is easy to observe that $|\tau|>1$. Now consider $\psi(\lambda)=(\lambda-z)^{-1}$, $\lambda\in \mathbb{R}$, and $\phi(w)=\dfrac{i(1+\tau)^2}{2}\left[\dfrac{1}{1+\tau}+\dfrac{1}{w-\tau}\right]$, $w\in \mathbb{T}$. Then it is easy to conclude that $\psi(\lambda)=\phi\left(\dfrac{i-\lambda}{i+\lambda}\right)$, $\phi\in \mathcal{A}_{\mathbb{T}}$ and hence $\psi \in \mathcal{F}_\mathbb{R}$. Therefore by applying \eqref{selfuni} corresponding to $\psi$ we get
\begin{align}\label{eqq}
	\operatorname{Tr}\Bigg\{ (H-z)^{-1}-(H_0-z)^{-1}-\dfrac{d}{ds}\Big|_{s=0}(W_s-z)^{-1}\Bigg\}
	=\int_{-\infty}^{\infty}~\dfrac{2(1+\lambda z)}{(\lambda-z)^3} ~\xi(\lambda)~d\lambda.
\end{align}
On the other hand the equality $\phi(U_0)=\psi(H_0)$ yields that
\begin{align*}
	(U_0-\tau)^{-1}=~\frac{i}{2}(i+z)\dfrac{i+H_0}{H_0-z},
\end{align*}
and hence
\begin{align}\label{dereq}
	\nonumber\dfrac{d}{ds}\Big|_{s=0}(W_s-z)^{-1}=&\dfrac{i(1+\tau)^2}{2}\at{\dds}{s=0}(U_s-\tau)^{-1}\\
	\nonumber	=&-2i~\dfrac{i(1+\tau)^2}{2} (U_0-\tau)^{-1}(R_{-i}-R^0_{-i})(U_0-\tau)^{-1}\\
	=& \frac{i+H_0}{z-H_0}~~M~~\frac{i+H_0}{z-H_0},
\end{align}
where $M=R_{-i}-R^0_{-i}$. Therefore combining equations \eqref{eqq} and \eqref{dereq} we get \eqref{res}. This completes the proof. 
\end{proof}
\begin{rem}
It is important to note that even though the formulas \eqref{selfuni} and \eqref{res} are look like similar to the formulas (3.12) and (3.11) respectively obtained in \cite[Theorem 3.2]{NH}, but they are actually different since our path $W_s=\Big((H+i)(H_s+i)^{-1}(H_0+i)-i\Big)$ is not exactly same as the path considered by Neidhardt in \cite[Theorem 3.2]{NH}. In other words, we have obtained the path $W_s$ by considering the Cayley  transformation on the linear path associated to the pair $(U,U_0)$ whereas Neidhardt obtained the required path by considering the Cayley transformation on the multiplicative path associated to the pair $(U,U_0)$. Furthermore, the path $W_s$ is closer to the path considered by Koplienko \cite{Ko}. Also, we would like to mention that using the multiple operator integrals (MOI), the trace formula \eqref{selfuni} was obtained in \cite{PoSkSu15, Sk}, but here we present the formula using one our main results, namely Theorem \ref{th5} without using MOI, particularly focusing on the path. Our motive in keeping the details of the very short proof of the above theorem is to keep this article self-contained and to aid the study of the next section.
\end{rem}
\section{\bf Trace formula for maximal dissipative operators}
In this section, our aim is to obtain the Koplienko trace formula for a pair of maximal dissipative operators from the existing Koplienko trace formula \eqref{eqin20} corresponding to a pair of contractions $(T,T_0)$ such that $T-T_0\in \mathcal{B}_2(\mathcal{H})$. In this connection, it is worth mentioning a paper by Malamud, Neidhardt and Peller \cite{MNP19}, where an analogous study of Krein's trace formula for a pair of maximal dissipative operators was achieved. To that aim we start with the following definition.
\begin{defn}
A densely defined linear operator $L$ in $\hil$ is called dissipative if $\operatorname{Im}\langle Lh, h\rangle\geq 0$ for $h\in Dom(L)$. It is called maximal dissipative if $L$ has no proper dissipative extension.
\end{defn}
The Cayley transform of a maximal dissipative operator $L$ is defined by
\begin{align}
T=(i-L)(i+L)^{-1}.
\end{align} 
It is well known that $T$ is a contraction. Moreover, a contraction $T$ is the Cayley transform of a maximal dissipative operator $L$ if and only if $1$ is not an eigenvalue of $T$. Now we have the following main theorem in this section.

\begin{thm}\label{dissipthm}
Let $L$ and $L_0$ be two maximal dissipative operators in a separable infinite dimensional Hilbert space $\hil$ such that $Dom(L) = Dom(L_0)$ and  $(L+i)^{-1}-(L_0+i)^{-1}\in\hils$. Let $$Q_s=\Big((L+i)(L_s+i)^{-1}(L_0+i)-i\Big),$$ where $L_s=sL_0+(1-s)L,~s\in [0,1] $. Then there exists a measurable function $\xi:\mathbb{R}\to\mathbb{C}$ obeying $(1+\lambda^2)^{-1}\xi\in L^1(\mathbb{R})$ such that
\begin{align}
	\operatorname{Tr}\left\{\psi(L)-\psi(L_0)-\dfrac{d}{ds}\Big|_{s=0}\psi(Q_s)\right\}=\int_{-\infty}^{\infty}~\dfrac{d}{d\lambda}\Big\{(1+\lambda^2)\psi'(\lambda)\Big\}~\xi(\lambda)~d\lambda 
\end{align} 
for each $\psi\in\mathcal{F}_\mathbb{R}$.
\end{thm}
\begin{proof}
Let $T$ and $T_0$ be the Cayley transform of $L$ and $L_0$ respectively, that is
\begin{align*}
	T=(i-L)(i+L)^{-1} \text{ and } T_0=(i-L_0)(i+L_0)^{-1}.
\end{align*}
Consequently,
\begin{align*}
	L=i(1-T)(1+T)^{-1} \text{ and } L_0=i(1-T_0)(1+T_0)^{-1}.
\end{align*}
It is easy to observe that $(i-Q_s)(i+Q_s)^{-1}=T_s$. Let $\psi\in\mathcal{F}_{\mathbb{R}}$. Then there exists $\phi\in\mathcal{A}_\cir$ such that $\psi(\lambda)=\phi\left(\dfrac{i-\lambda}{i+\lambda}\right)$ and hence $\psi(L)=\phi(T)$, $\psi(L_0)=\phi(T_0)$ and $\psi(Q_s)=\phi(T_s)$. Therefore by similar kind of argument as in Theorem \ref{selfth} and using Theorem~\ref{PSthm} we conclude that
there exists a measurable function $\xi:\mathbb{R}\to\mathbb{C}$ obeying $(1+\lambda^2)^{-1}\xi\in L^1(\mathbb{R})$ such that
\begin{align*}
	\operatorname{Tr}\left\{\psi(L)-\psi(L_0)-\dfrac{d}{ds}\Big|_{s=0}\psi(Q_s)\right\}=\int_{-\infty}^{\infty}~\dfrac{d}{d\lambda}\Big\{(1+\lambda^2)\psi'(\lambda)\Big\}~\xi(\lambda)~d\lambda.
\end{align*} 
This completes the proof. 
\end{proof}

\section{\bf Extension of Koplienko-Neidhardt trace formula}
In this section, we deal with the extension of Koplienko-Neidhardt trace formula in the following sense: Neidhardt obtained the formula \eqref{intequ4} corresponding to the pair $(U_0,A)$, where $U_0$ is a unitary operator on $\mathcal{H}$ and $A=A^*\in \mathcal{B}_2(\mathcal{H})$. In this regard, the following theorem deals with the formula \eqref{intequ4} corresponding to the pair $(T_0,A)$, where $T_0=N_0+V$ is a contraction on $\mathcal{H}$ such that $N_0$ is a bounded normal operator on $\mathcal{H}$, $V\in \mathcal{B}_2(\mathcal{H})$ and $A=A^*\in \mathcal{B}_2(\mathcal{H})$. We use finite dimensional approximation method as earlier to prove our result.  In this regard, we would also like to mention that the Krein and Koplienko-Neidhardt trace formulas for certain pair of contractions in the multiplicative path was considered in \cite[Lemma 2.1]{MNP19} and \cite{MaMo} respectively.
\begin{thm}\label{main}
Let $T_0=N_0+V$ be a contraction in an infinite dimensional separable Hilbert space $\hil$ such that $N_0^*N_0=N_0N_0^*$ and $V\in\hils$. Let $A=A^*\in\hils$. Denote $T_s=e^{isA}T_0, ~s\in[0,1],$ and $T=T_1$. Then for any complex polynomial  $p(\cdot)$,
$\Bigg\{p(T)-p(T_0)-\at{\dds}{s=0} \big\{p(T_s)\big\} \Bigg\}\in\boh $
and there exists an $L^1(\mathbb{T})$-function $\tilde{\eta}$  (unique up to an analytic term) such that
\begin{align}\label{eq201}
	\operatorname{Tr}\Bigg\{p(T)-p(T_0)-\at{\dds}{s=0} \big\{p(T_s)\big\} \Bigg\}=\int_{0}^{2\pi}
	\dfrac{d^2}{dt^2}\Big(p(e^{it})\Big)\tilde{\eta}(t)dt.
\end{align}
Moreover, for every given $\epsilon >0$, we choose the function $\tilde{\eta}$ satisfying \eqref{eq201} in such a way so that 
\begin{equation}\label{eq35}
	\|\tilde{\eta}\|_{L^1(\mathbb{T})}
	\leq (\frac{1}{2}+\epsilon) ~\|A\|_2^2.
\end{equation}
\end{thm}
\begin{proof} 
Using Theorems \ref{th3(2)} and \ref{th2(2)}, we have that
\begin{align}\label{limequal}
	\nonumber \operatorname{Tr}\Big\{p(T)-p(T_0)-\at{\dds}{s=0}p(T_s) \Big\}=&\lim_{n\to \infty}\operatorname{Tr}~\left[P_n\Big\{p(T_n)-p(T_{0,n})-\at{\dds}{s=0}p(T_{s,n}) \Big\}P_n\right]\\
	=&\lim\limits_{n\to \infty}\int_{0}^{2\pi}\dfrac{d^2}{dt^2}\Big(p(e^{it})\Big)\tilde{\eta}_n(t)dt,
\end{align}
where $A_n=P_nAP_n$, $T_{0,n} =P_nT_0P_n$,  $T_n=e^{iA_n}T_{0,n}$,  $T_{s,n}=e^{isA_n}T_{0,n}$, and 
\begin{equation}\label{aeqab}
	\tilde{\eta}_n(t)=\int_{0}^{1}\operatorname{Tr} \Big[A_n\Big\{\mathcal{F}_{0,n}(t)-\mathcal{F}_{s,n}(t)\Big\}\Big]ds, ~t\in[0,2\pi],
\end{equation}
where $\mathcal{F}_{0,n}(\cdot),~\mathcal{F}_{s,n}(\cdot)$ are corresponding semi-spectral measures of the contractions $T_{0,n}$ and $T_{s,n}$ respectively.
Moreover, from \eqref{etabdd} it follows that 
\begin{equation}\label{eq36Final}
	\big\|[\tilde{\eta}_n]\big\|_{L^1(\mathbb{T})/H^1(\mathbb{T})}
	\leq \frac{1}{2} ~\|A_n\|_2^2.
\end{equation}
Next we show that the sequence $\{\tilde{\eta}_n\}$ converges in some suitable sense. Indeed, using the similar setup as in the proof of the Theorem \ref{th4}, we have for $f\in \mathcal{P}(\mathbb{T})$ and $g(e^{it})=\int_{0}^{t}f(e^{is})ie^{is}ds$, $t\in[0,2\pi]$ that 
\begin{align}\label{eqq31}
	\nonumber &\int_{\cir} f(z)\big\{\tilde{\xi}_n(z)-\tilde{\xi}_m(z)\big\} dz
	=\int_{0}^{2\pi} f(e^{it})\big\{\tilde{\eta}_n(e^{it})-\tilde{\eta}_m(e^{it})\big\}~ie^{it}dt\\
	&=\int_{0}^{2\pi}\ddt \big\{g(e^{it})\big\}\Bigg[\int_{0}^{1}\operatorname{Tr} \Big[A_n\Big\{\mathcal{F}_{0,n}(t)-\mathcal{F}_{s,n}(t)\Big\}-A_m\Big\{\mathcal{F}_{0,m}(t)-\mathcal{F}_{s,m}(t)\Big\}\Big]~ds\Bigg]dt,
\end{align}
where we set $\tilde{\xi}_n(z)=\tilde{\eta}_n(Arg(z)),~z\in\cir$. 
Next by using Fubini's theorem to interchange the orders of integration and integrating by-parts, the above expression \eqref{eqq31} becomes
\begin{align*}\label{eqqq31}
	\nonumber &\int_{\cir} f(z)\big\{\tilde{\xi}_n(z)-\tilde{\xi}_m(z)\big\} dz\\
	\nonumber &=\int_{0}^{1}ds~\operatorname{Tr} \Bigg[A_n\Bigg\{\Big\{g(T_{s,n})-g(T_{s})\Big\}-\Big\{g(T_{0,n})-g(T_{0})\Big\}\Bigg\}\\
	&\hspace*{.3in}-A_m\Bigg\{\Big\{g(T_{s,m})-g(T_{s})\Big\}-\Big\{g(T_{0,m})-g(T_{0})\Big\}\Bigg\}+(A_n-A_m)\Big\{g(T_{s})-g(T_{0})\Big\}\Bigg],
\end{align*}
which by using Theorem \ref{estimatethm} yields
\begin{equation*}\label{ap}
	\begin{split}
		&\left|\int_{\cir} f(z)\big\{\tilde{\xi}_n(z)-\tilde{\xi}_m(z)\big\} dz\right|\leq~~\int_{0}^{1}\Big[\|A\|_2\|f\|_\infty\Big[\left\|T_{s,n}P_n-P_nT_s\right\|_2 +\left\|T_{0,n}P_n-P_nT_0\right\|_2\\
		&+\left\|T_{s,m}P_m-P_mT_s\right\|_2+\left\|T_{0,m}P_m-P_mT_0\right\|_2 \Big]+\|f\|_\infty\|(A_n-A_m)\|_2\left\|T_{s}-T_{0}\right\|_2\Big]~ds,
	\end{split}
\end{equation*}
and hence by applying Lemma \ref{supnorm} we have
\begin{align}\label{eqe}
	\nonumber&\big\|[\tilde{\xi}_n]-[\tilde{\xi}_m]\big\|_{L^1(\mathbb{T})/H^1(\mathbb{T})}\\
	\nonumber=&\big\|[\tilde{\xi}_n-\tilde{\xi}_m]\big\|_{L^1(\mathbb{T})/H^1(\mathbb{T})}=~\sup_{\substack{f\in \mathcal{P}(\cir);~\|f\|_{\infty}\leq 1}}\Bigg|\int_{\cir} f(z)\big\{\tilde{\xi}_n(z)-\tilde{\xi}_m(z)\big\}dz\Bigg|\\
	\nonumber\leq&~~\int_{0}^{1}\Big[\|A\|_2\left[\left\|T_{s,n}P_n-P_nT_s\right\|_2 +\left\|P_nT_0P_n^\perp\right\|_2+\left\|T_{s,m}P_m-P_mT_s\right\|_2+\left\|P_mT_0P_m^\perp\right\|_2 \right]\\
	&\hspace*{2in}+\|(A_n-A_m)\|_2\left\|T_{s}-T_{0}\right\|_2\Big]~ds.
\end{align}
Finally, using the estimates listed in Lemma~\ref{appthm} we conclude that the right hand side of \eqref{eqe} converges to zero as $m,n\longrightarrow \infty$. Therefore $\Big\{[\tilde{\xi}_n]\Big\}$ is a Cauchy sequence in $L^1(\mathbb{T})/H^1(\mathbb{T})$ and hence there exists a $\tilde{\xi} \in L^1(\mathbb{T})$ such that $\Big\{[\tilde{\xi}_n]\Big\}$ converges to $[\tilde{\xi}]$ in $L^1(\mathbb{T})/H^1(\mathbb{T})$-norm which by using \eqref{limequal} yields
\begin{align*}
	\operatorname{Tr}\Big\{p(T)-p(T_0)-\at{\dds}{s=0}p(T_s) \Big\}= \int_{0}^{2\pi}\dfrac{d^2}{dt^2} \left\{p(e^{it})\right\}\tilde{\xi}(e^{it})dt=\int_{0}^{2\pi}\dfrac{d^2}{dt^2} \left\{p(e^{it})\right\}\tilde{\eta}(t)dt,
\end{align*}
where $\tilde{\eta}(t)=\tilde{\xi}(e^{it})~t\in[0,2\pi]$.
Moreover, using the estimate \eqref{eq36Final} and applying the definition of the $L^1(\mathbb{T})/H^1(\mathbb{T})$- norm we conclude 
the estimate \eqref{eq35}. This completes the proof. 
\end{proof}

\begin{rem}
By repeating the similar argument the above Theorem \ref{main} can also be extended to the class $\mathcal{F}_2^+(\cir)$.
\end{rem}

\section*{\bf Acknowledgement}
  The authors are extremely grateful to Prof. Fritz Gesztesy for his valuable comments, which help to write the introduction part more visibly. A. Chattopadhyay is supported by the Core Research Grant (CRG), File No: CRG/2023/004826, by the Science and Engineering Research Board (SERB), Department of Science \& Technology (DST), Government of India. S. Das acknowledges financial support from the Indian Statistical Institute Bangalore, India. C. Pradhan acknowledges support from JCB/2021/000041 as well as IoE post-doctoral fellowship from the Indian Institute of Science Bangalore, India. We sincerely thank the referees for several important suggestions which considerably improved the presentation of this paper.

\section*{\bf Declarations} 	

{\textbf{Conflicts of interests:}}	The authors declare that they have no conflict of interest. No datasets were generated or analyzed during the current study.

\end{document}